\documentclass[11pt]{amsart}
\usepackage{mathrsfs}
\usepackage{amssymb,color}
\usepackage{amsmath}
\usepackage{amsthm}
\usepackage{lineno}
\makeatletter
\def\LaTeX{\leavevmode L\raise.42ex
    \hbox{\kern-.3em\size{\sf@size}{0pt}\selectfont A}\kern-.15em\TeX}
\makeatother

\newcommand{\BibTeX}{{\rm B\kern-.05em{\sc
          i\kern-.025emb}\kern-.08em\TeX}}

\newcommand{\lbl}[1]{\label{#1}}

\newtheorem{theo}{Theorem}[section]

\newtheorem{lem}[theo]{Lemma}
\newtheorem{remark}[theo]{Remark}
\newtheorem{definition}[theo]{Definition}
\newtheorem{prop}[theo]{Proposition}

\newcommand\ee{\end{equation}}
\newcommand\bes{\begin{eqnarray}}
\newcommand\ees{\end{eqnarray}}
\newcommand\bess{\begin{eqnarray*}}
\newcommand\eess{\end{eqnarray*}}

\begin{document}
\title[The weak competition system with a free boundary]
{asymptotic spreading speed for the weak competition system with a free boundary}

\author[Z.G. Wang, H. Nie and Y. Du]{Zhiguo Wang$^{\dag}$, Hua Nie$^{\dag}$, Yihong Du$^{\ddag,\ast}$}

\thanks{$^\dag$ School of Mathematics and Information Science,
Shaanxi Normal University, Xi'an, Shaanxi {\rm710119,} China.}

\thanks{$^\ddag$ School of Science and Technology,
University of New England, Armidale, NSW {\rm2351,} Australia.}

\thanks{$^\ast$The corresponding author. E-mail address: ydu@turing.edu.au (Y. Du).}



\date{\today}

\begin{abstract} \small This paper is concerned with a diffusive Lotka-Volterra type competition system with a free boundary in one space dimension. Such a system may be used to describe the invasion of a new species into the habitat of a native competitor.
We show that the long-time dynamical behavior of the system is determined by a spreading-vanishing dichotomy, and provide sharp criteria for spreading and vanishing of the invasive species. Moreover, we determine the asymptotic spreading speed of the invasive species when its spreading is successful, which involves two systems of traveling wave type equations, and is highly nontrivial to establish.
\end{abstract}
\maketitle

\section{Introduction}
{\setlength{\baselineskip}{16pt}{\setlength\arraycolsep{2pt}
The classical Lotka-Volterra reaction-diffusion system
\begin{linenomath*}\bes\left\{
\begin{aligned}
&u_t=d_1u_{xx}+u(a_1-b_1u-c_1v),&x\in\mathbb{R},&\quad t>0,\\
&v_t=d_2v_{xx}+v(a_2-b_2v-c_2u),&x\in\mathbb{R},&\quad t>0
\end{aligned}
\right.\lbl{section1-a1}\ees\end{linenomath*}
is a model frequently used to describe competitive behavior between two distinct species.
Here $u(x, t)$ and $v(x, t)$ denote the population densities of two competing species at the position
$x$ and time $t$; the constants $d_i, a_i,b_i$ and $c_i\; (i=1,2)$ are the diffusion rates, intrinsic growth rates, intra-specific competition rates, and inter-specific competition rates, respectively,  all of which are assumed to be positive. By setting
$$\begin{aligned}
\hat{u}(x,t):=\frac{b_1}{a_1}u\left(\sqrt{\frac{d_1}{a_1}}x, \frac{t}{a_1}\right),
&\quad\hat{v}(x,t):=\frac{b_2}{a_2}v\left(\sqrt{\frac{d_1}{a_1}}x, \frac{t}{a_1}\right),\\
d:=\frac{d_2}{d_1},\quad r:=\frac{a_2}{a_1},&\quad k:=\frac{a_2c_1}{a_1b_2},\quad h:=\frac{a_1c_2}{a_2b_1},
\end{aligned}$$
and dropping the hat signs, system \eqref{section1-a1} becomes the following nondimensional system:
\bes\left\{
\begin{aligned}
&u_t=u_{xx}+u(1-u-kv),&x\in\mathbb{R},\quad t>0,\\
&v_t=dv_{xx}+rv(1-v-hu),&x\in\mathbb{R},\quad t>0.
\end{aligned}
\right.\lbl{section1-a2}\ees
It is easy to see that \eqref{section1-a2} has four
equilibria: $(0,0),(1,0),(0,1)$ and $(u^\ast,v^\ast)=\left(\frac{1-k}{1-hk},\frac{1-h}{1-hk}\right)$, with $(u^*, v^*)$ meaningful only when  $(1-k)(1-h)>0$. When the entire real line $\mathbb{R}$ is replaced
by a bounded open interval in $\mathbb{R}$, under the zero Neumann boundary conditions, the asymptotic behavior of the solution $(u(x, t), v(x, t))$
for \eqref{section1-a2} with initial functions $u(x, 0), v(x, 0) > 0$ can be summarized below (see, for example \cite{Mo}):
 \begin{itemize}
   \item [(\textrm{I})] if $k<1<h$, then $\lim_{t\rightarrow\infty}(u(x, t), v(x, t))=(1,0)$;
   \item [(\textrm{II})] if $h<1<k$, then $\lim_{t\rightarrow\infty}(u(x, t), v(x, t))=(0,1)$;
   \item [(\textrm{III})] if $h,k<1$, then $\lim_{t\rightarrow\infty}(u(x, t), v(x, t))=(u^\ast,v^\ast)$;
   \item [(\textrm{IV})]if $h,k>1$, then $\lim_{t\rightarrow\infty}(u(x, t), v(x, t))=(1,0)$ or $(0,1)$ or $(u^*,v^*)$ (depending on the initial condition).
 \end{itemize}
The cases  (I) and (II) are usually called the weak-strong competition case, while (\textrm{III}) and (\textrm{IV}) are known as the  weak and strong competition cases, respectively.

A number of variations of \eqref{section1-a2} (or \eqref{section1-a1}) have been used to model the spreading of a new or invasive species.
For example, to describe the invasion of a new species into the habitat of a native competitor, Du and Lin \cite{DL14} considered the following
 free boundary problem
\bes\left\{
\begin{aligned}
&u_t=u_{xx}+u(1-u-kv),&&0<x<g(t),&\quad t>0,&\\
&v_t=dv_{xx}+rv(1-v-hu),&&0<x<\infty,&\quad t>0,&\\
&u_x(0,t)=v_x(0,t)=0,\quad u(x,t)=0,&&g(t)\leq x<\infty, &\quad t>0,&\\
&g'(t)=-\gamma u_x(g(t),t),&&&t>0,&\\
&g(0)=g_0,\quad u(x,0)=u_{0}(x),&&0\leq x\leq g_0,& &\\
&v(x,0)=v_{0}(x),&& 0\leq x<\infty,& &
\end{aligned}
\right.\lbl{section1-1}\ees
where $x=g(t)$ is usually called a free boundary, which is to be determined together with $u$ and $v$. The initial functions satisfy
\bes
\left\{
\begin{aligned}
&u_0\in C^2([0,g_0]),\; u_0'(0)=u_0(g_0)=0\mbox{ and }u_0(x)>0\mbox{ in }[0,g_0),\\
&v_0\in C^2([0,\infty))\cap L^\infty(0,\infty),\; v_0'(0)=0,\; \liminf_{x\rightarrow\infty}v_0(x)>0\mbox{ and }v_0(x)> 0\mbox{ in }[0,\infty).
\end{aligned}
\right.\lbl{section1-1a}\ees
This model describes how a new species with population density $u$ invades into the habitat of a native competitor $v$. It is assumed that the species $u$ exists initially in the range $0<x<g_0$, invades into  new territory through its invading front  $x=g(t)$. The native species $v$  undergoes diffusion and growth in the available habitat $0<x<\infty$. Both $u$ and $v$ obey a no-flux
boundary condition at $x=0$. The equation $g'(t)=-\gamma u_x(g(t),t)$ means that the invading speed  is proportional to the gradient of the population density of $u$  at the invading front, which coincides with the well-known Stefan free bounary condition.  All parameters $d,k,h,r,g_0$ and $\gamma$ are assumed to be positive. For  more  biological background, we refer to \cite{BDK, DG,DL10,DL14}.

The work \cite{DL14} considers the {\bf weak-strong competition} case only. It is shown in \cite{DL14} that when the invasive species $u$ is the inferior competitor $(k>1>h)$, if the resident species $v$ is already well established initially (i.e., $v_0$ satisfies the conditions in \eqref{section1-1a}), then  $u$ can never invade deep into the underlying habitat, and it dies out before its invading front reaches a certain finite limiting position, whereas if the invasive species $u$ is superior $(h>1>k)$, a spreading-vanishing dichotomy holds for $u$ (see Theorem 4.4 in \cite{DL14}). Moreover, when spreading of $u$ happens, the precise asymptotic spreading speed has been given by Du, Wang and Zhou \cite{DWZ}; it concludes that the spreading speed of $u$ has an asymptotic limit as time goes to infinity, which is determined by a certain traveling wave type system.

In this paper, we examine the {\bf weak competition} case of \eqref{section1-1}, namely the case
$$0<k<1,\; 0<h<1.$$
We will show that a similar spreading-vanishing dichotomy holds for the invasive species $u$, but in sharp contrast to the weak-strong competition case $(h>1>k)$ in \cite{DL14}, where when $u$ spreads successfully, $v$ vanishes eventually (namely $(u,v)\to (1,0)$ as $t\to\infty$), here in the weak competition case, when $u$ spreads successfully, the two populations converge to the co-existence steady state $(u^*, v^*)$ as time goes to infinity. In fact, our results here indicate that the native competitor $v$ always
survives the invasion of $u$. Moreover, we also determine the precise spreading speed of $u$ when the invasion is successful,
which turns out to be the most difficult part of this work and consititutes the main body of the paper. We would like to stress that
while the main steps in the approach  here are similar in spirit to those in \cite{DL14} and \cite{DWZ}, highly nontrivial  changes are needed in the detailed techniques,
due to the different nature of the dynamical behavior of the system under the current weak competition assumption.

\medskip

We now state our main results more precisely. From \cite{DL14} we know that \eqref{section1-1} has a unique solution, which is defined for all $t>0$. Our aim here is to determine its long-time behavior.

\begin{theo}\lbl{f2}Suppose that $h,k\in (0,1)$ and $(u,v,g)$ is the solution of \eqref{section1-1} with $u_0$ and $v_0$ satisying \eqref{section1-1a}. Then, as $t\to\infty$, the following dichotomy holds.

Either {\bf (i) the species $u$ spreads successfully:}
\[\mbox{
$\lim_{t\to\infty}g(t)=\infty$ and \;\; $
\lim_{t\to\infty}(u(\cdot,t),v(\cdot,t))= (u^\ast,v^\ast)$  in $C^2_{loc}([0,\infty))$;}
\]

or
 {\bf (ii) the species $u$ vanishes eventually:}
\[\mbox{ $\lim_{t\to\infty}g(t)<\infty$ and\;\; $
\lim_{t\to\infty}(u(\cdot,t),v(\cdot,t))= (0,1)$  in $C^2_{loc}([0,\infty))$.}
\]
\end{theo}

The next theorem provides a sharp criterion for the above spreading-vanishing dichotomy.

\begin{theo}\lbl{f3}   Under the assumptions of Theorem \ref{f2}, there exists $\gamma^*\in [0,\infty)$ depending on $(u_0, v_0)$ such that
alternative {\rm (i)} in Theorem \ref{f2} happens if and only if $\gamma>\gamma^*$. Moreover,
$\gamma^*=0$ $($and hence  $u$ always spreads successfully$)$ if $g_0\geq \frac{\pi}{2\sqrt{1-k}}$, and $\gamma^*>0$ if
 $g_0<\frac{\pi}{2}$.
\end{theo}

When case (i) happens in Theorem \ref{f2}, the spreading speed of $u$ is asymptotically linear, as indicated by the following theorem.

\begin{theo}\lbl{F5} If {\rm (i)} happens in Theorem \ref{f2}, then there exists $c^0>0$ such that
\bess\lim_{t\rightarrow\infty}\frac{g(t)}{t}=c^0.\eess
\end{theo}

As usual, the positive constant $c^0$ in Theorem \ref{F5} is called the {\bf asymptotic spreading speed} of $u$.
The key in the proof of this theorem is to find a way to determine $c^0$. It turns out that two systems of traveling
wave type equations are needed in order  to determine $c^0$. The first one is obtained by looking for traveling wave solutions of \eqref{section1-a2},
namely
\bes
\begin{cases}
\Phi''-c\Phi'+\Phi(1-\Phi-k\Psi)=0,\quad\Phi'>0,\quad -\infty<s<\infty,\\
d\Psi''-c\Psi'+r\Psi(1-\Psi-h\Phi)=0,\quad\Psi'<0,\quad -\infty<s<\infty,\\
(\Phi,\Psi)(-\infty)=(0,1),\quad (\Phi(\infty),\Psi(\infty))=(u^\ast,v^\ast).
\end{cases}\lbl{entirewave0}\ees
(The second system is \eqref{section12} below.)

Clearly, if $(\Phi(s),\Psi(s))$ solves \eqref{entirewave0}, then
$$(u(x,t),v(x,t)):=(\Phi(ct-x),\Psi(ct-x))$$
is a solution of \eqref{section1-a2}, which is often called a traveling wave solution with speed $c$.

By Theorem 4.2 and Example 4.2 in \cite{LWL}, we have the following result on \eqref{entirewave0}:

\begin{prop}\lbl{F6}Assume $h, k\in (0,1)$. Then there exists a critical speed $c_{\ast}\geq 2\sqrt{1-k}$ such that
 \eqref{entirewave0} has a solution when $c\geq c_{\ast}$ and it has no solution when $c<c_{\ast}$.
\end{prop}

We can further show that $c_{\ast}\leq 2\sqrt{u^\ast}$. Making use of $c_*$, we have the following result
on the system below which gives traveling wave type solutions to \eqref{section1-1}:
\bes\left\{
\begin{aligned}
&\phi''-c\phi'+\phi(1-\phi-k\psi)=0,\quad \phi'>0\quad\mbox{for }  0<s<\infty,\\
&d\psi''-c\psi'+r\psi(1-\psi-h\phi)=0,\quad \psi'<0\quad\mbox{for }  -\infty<s<\infty,\\
&\phi(s)\equiv0\mbox{ for }s\leq0,\quad\psi(-\infty)=1,\quad(\phi(\infty),\psi(\infty))=(u^\ast,v^\ast).&\\
\end{aligned}
\right.\lbl{section12}\ees

\begin{theo}\lbl{F4}
Assume $h, k\in (0,1)$. If $c\in[0,c_{\ast})$, then the system \eqref{section12} admits a unique solution $(\phi_c,\psi_c)\in [C(\mathbb{R})\cap C^2(\mathbb{R}^+)]\times C^2(\mathbb{R})$; if $c\geq c_{\ast}$, then  \eqref{section12} does not have a  solution.
Moreover, the following conclusions hold:
\begin{itemize}
\item[\rm{(i)}] $\lim_{c\nearrow c_{\ast}}(\phi_c,\psi_c)=(0,1)\mbox{ in }C_{loc}^2([0,\infty))\times C_{loc}^2(\mathbb{R}),$
\item[\rm{(ii)}] for any $\gamma>0$, there exists a unique $c_\gamma\in(0, c_\ast)$ such that $\gamma \phi_{c_\gamma}'(0)=c_\gamma$,
\item[\rm{(iii)}]the function $\gamma\longmapsto c_\gamma$ is strictly increasing and $\lim_{\gamma\rightarrow\infty}c_\gamma= c_\ast.$
\end{itemize}
\end{theo}

We will show that the asymptotic spreading speed of $u$ in Theorem \ref{F5} is given by
\[
c^0:=c_\gamma.
\]
Let us note that if $(\phi,\psi, c)$ solves \eqref{section12}, then
\bess\tilde{u}(x,t):=\phi(ct-x),\quad\tilde{v}(x,t):=\psi(ct-x)\eess
 satisfy
\bess\left\{
\begin{aligned}
&\tilde{u}_t=\tilde{u}_{xx}+\tilde{u}(1-\tilde{u}-k\tilde{v}),&-\infty<x<ct,\quad t>0,\\
&\tilde{v}_t=d\tilde{v}_{xx}+r\tilde{v}(1-\tilde{v}-h\tilde{u}),&-\infty<x<\infty,\quad t>0,\\
&\tilde{u}(x,t)=0,&\quad ct\leq x<\infty, \quad t>0.
\end{aligned}
\right.\lbl{section1-01}\eess
If $c=c_\gamma$, then we have additionally
$$(ct)'=c=-\gamma\tilde{u}_{x}(ct,t).$$
We call $(\phi_c,\psi_c)$ with $c=c_\gamma$  the \emph{\textbf{semi-wave}} associated with \eqref{section1-1}. This pair of functions (and its suitable variations) will play a crucial role in the proof of Theorem \ref{F5}, and they also provide upper and lower bounds for the solution pair $(u,v)$ (see the proof of Lemmas \ref{T15} and \ref{T14} for details).

Related free boundary problems of two-species competition models have been investigated in many recent works. Apart from \cite{DL14, DWZ} mentioned earlier, one may find various interesting results in
 \cite{CLW,GW12,GW15,WJ,WJZ,WZ16,WZ14,Wu13,Wu15,ZW16}, where \cite{CLW,WZ16} considers time-periodic environment, \cite{WJ,WJZ,ZW16} considers space heterogeneous environment, \cite{GW12,Wu13,Wu15} considers the weak competition case and \cite{GW15,WZ14} covers more general situations.
However, except \cite{DWZ}, in all these works the question of whether there is a precise asymptotic spreading speed has been left open. This is in sharp contrast to the corresponding one species models, where the precise spreading speed is obtained in many situations; see, for example, \cite{DDL, DG, DGP, DLiang, DL10,  DLe, DMZ, DMZ2, LLS}.

The rest of this paper is organized as follows.
In section 2, we present some basic results including the existence of solutions to \eqref{section1-1}, and the existence of solutions to a more general system than \eqref{section12}.
 In section 3, we prove Theorem \ref{F4} based on an upper and lower solution result (Proposition \ref{s7}) established in section 2, and many other techniques. In section 4, we investigate the long time behavior of the solution to \eqref{section1-1}
and prove the spreading-vanishing dichotomy,  and obtain sharp criteria for spreading and vanishing, which finish the proof of Theorems \ref{f2} and \ref{f3}. In section 5, we complete the proof of Theorem \ref{F5} by making use of Theorem \ref{F4}.
Section 6 consists of the proof of Proposition \ref{s7} stated in section 2.

{\setlength{\baselineskip}{16pt}{\setlength\arraycolsep{2pt} \indent
\section{Preliminary results }
\setcounter{equation}{0}

In this section, we collect some basic facts which will be needed in our proof of the main results.
We first note that \eqref{section1-1} always has a unique  solution. Indeed, by Theorems 2.4 and 2.5 in \cite{DL14}, we have the following  results on the solution of system \eqref{section1-1}.

\begin{prop}\label{t1} For any initial function $(u_0, v_0)$ satisfying \eqref{section1-1a}, the free boundary problem \eqref{section1-1} admits a unique solution
$$(u,v,g)\in C^{1+\alpha,(1+\alpha)/2}(D_1)\times C^{1+\alpha,(1+\alpha)/2}(D_2)\times C^{1+\alpha/2}([0,\infty)),$$
where $D_1=\{(x,t):x\in[0,g(t)],t\in[0,\infty)\}$ and $D_2=\{(x,t):x\in[0,\infty),t\in[0,\infty)\}$. Furthermore, there exists a positive constant $M$ depending on $d,r,\gamma,\|u_0\|_\infty$ and $\|v_0\|_\infty$, such that
\bess\begin{aligned}
&0<u(x,t)\leq M,\; 0< g'(t)\leq M\mbox{ for } 0<x<g(t),\quad t>0,\\
&0<v(x,t)\leq M\mbox{ for } 0<x<\infty,\quad t>0.
\end{aligned}\eess
\end{prop}

Next we recall   a comparison result for the free boundary problem, which  is a special case of Lemma 2.6 in \cite{DL14}.
\begin{prop}\lbl{the comparison principle}\,{\rm (Comparison principle)}
Assume that $T\in(0,\infty)$, $\underline{g},\overline{g}\in C^1([0,T])$, $\underline{u}\in C(\overline{D_T^\ast})\cap C^{2,1}(D_T^\ast)$
with $D_T^\ast=\{(x,t)\in\mathbb{R}^2:x\in(0,\underline{g}(t)),t\in(0,T]\}$, $\overline{u}\in C(\overline{D_T^{\ast\ast}})\cap C^{2,1}(D_T^{\ast\ast})$
with $D_T^{\ast\ast}=\{(x,t)\in\mathbb{R}^2:x\in(0,\overline{g}(t)),t\in(0,T]\}$, $\underline{v},\overline{v}\in (L^\infty\cap C)([0,\infty)\times[0,T])\cap C^{2,1}([0,\infty)\times[0,T])$
and
\bess
\left\{
\begin{aligned}
&\displaystyle\overline{u}_t-\overline{u}_{xx}\geq \overline{u}(1-\overline{u}-k\underline{v}),&0<x<\overline{g}(t),&\quad0<t<T,\lbl{section4-1}\\
&\displaystyle \underline{u}_t-\underline{u}_{xx}\leq \underline{u}(1-\underline{u}-k\overline{v}),&0<x<\underline{g}(t),&\quad0<t<T,\lbl{section4-2}\\
&\displaystyle \overline{v}_t-d\overline{v}_{xx}\geq r\overline{v}(1-\overline{v}-h\underline{u}),&0<x<\infty,&\quad0<t<T,\lbl{section4-3}\\
&\displaystyle \underline{v}_t-d\underline{v}_{xx}\leq r\underline{v}(1-\underline{v}-h\overline{u}),&0<x<\infty,&\quad0<t<T,\lbl{section4-4}\\
&\overline{u}_x(0,t)\leq0,\quad\underline{v}_x(0,t)\geq0,\quad\overline{u}(x,t)=0,&
x=\overline{g}(t),&\quad0<t<T,\lbl{section4-5}\\
&\underline{u}_x(0,t)\geq0,\quad\overline{v}_x(0,t)\leq0,\quad\underline{u}(x,t)=0,&
x=\underline{g}(t),&\quad0<t<T,\lbl{section4-6}\\
&\underline{g}'(t)\leq-\gamma\underline{u}_x(\underline{g}(t),t),
\quad\overline{g}'(t)\geq-\gamma\overline{u}_x(\overline{g}(t),t),&&\quad0<t<T,\\
&\underline{u}(x,0)\leq u_{0}(x)\leq\overline{u}(x,0),&0\leq x\leq g_0,&\lbl{section4-9}\\
&\underline{v}(x,0)\leq v_{0}(x)\leq\overline{v}(x,0),&0\leq x<\infty.&
\end{aligned}
\right.\lbl{section4-10}
\eess
Then the solution $(u,v,g)$ of \eqref{section1-1} satisfies
$$\begin{aligned}
&\underline{g}(t)\leq g(t)\leq\overline{g}(t),\quad\underline{v}(x,t)\leq v(x,t)\leq\overline{v}(x,t)\; \mbox{ for }\; 0\leq x<\infty,\quad t\in[0,T],\\
&\underline{u}(x,t)\leq u(x,t)\mbox{ for }0\leq x<\underline{g}(t),\quad u(x,t) \leq \overline{u}(x,t) \;\mbox{ for }\; 0\leq x<g(t),\quad t\in[0,T].
\end{aligned}$$
\end{prop}

\begin{remark}\lbl{remarkp} In system \eqref{section1-1}, if the boundary conditions at $x=0$ are replaced by
\bess
u(0,t)=m_1(t),\quad v(0,t)=m_2(t),&t>0,
\eess
then Proposition \ref{the comparison principle} also holds if we replace
$$\overline{u}_x(0,t)\leq0,\quad \underline{v}_x(0,t)\geq0,\quad \underline{u}_x(0,t)\geq0,\quad \overline{v}_x(0,t)\leq0$$
by
$$\overline{u}(0,t)\geq m_1(t),\quad \underline{v}(0,t)\leq m_2(t),\quad \underline{u}(0,t)\leq m_1(t),\quad \overline{v}(0,t)\geq m_2(t)$$
for $0<t<T$.
\end{remark}

Lastly in this section, we modify some well known upper and lower solution technique to show the existence of a solution for a
general coorperative system of the form
\bes\begin{cases}
d_1\varphi_1''-c\varphi_1'+f_1(\varphi)=0,&s\in\mathbb{R}^+,\\
d_2\varphi_2''-c\varphi_2'+f_2(\varphi)=0,&s\in\mathbb{R},\\
\varphi_1(s)\equiv0,&s\leq0,
\end{cases}\lbl{function1}
\ees
where $\varphi=(\varphi_1,\varphi_2), c\geq0, d_i>0$, and $f_i:\mathbb{R}^2\rightarrow\mathbb{R}$  ($i=1,2$) satisfy the following conditions:

\begin{itemize}
\item[$(\textbf{A}_1)$] there is a strictly positive vector $\textbf{K}=(k_1,k_2)\in \mathbb{R}^2$ such that $f_i(\textbf{0})=f_i(\textbf{K})=0$ for $i\in\{1,2\}$, and $f_i(u_1,u_2)\not=0$ for $(u_1, u_2)\in \big((0,k_1]\times [0,k_2]\big)\setminus\{{\bf K}\}, i\in\{1,2\}$;
\item[$(\textbf{A}_2)$] the system \eqref{function1} is a cooperative system, that is, $f_i(u_1, u_2)$ is nondecreasing in $u_j$ for $i,j\in\{1,2\},i\neq j$, $(u_1, u_2)\in [0, k_1]\times [0, k_2]$, and there exists a constant $\beta_0\geq0$ such that $\beta_0 u_i+f_i(u_1, u_2)$ is  nondecreasing in $u_i$ for $ i=1,2$ and $(u_1, u_2)\in [0, k_1]\times [0, k_2]$;
\item[$(\textbf{A}_3)$] $f_1$ and $f_2$ are locally Lipschitz continuous.
\end{itemize}
We are particularly interested in solutions  $\varphi$ of \eqref{function1} that satisfy the asymptotic boundary conditions
\bes
\varphi_2(-\infty)=0,\quad \varphi(\infty)=\textbf{K}.
\lbl{function2a2}
\ees
Indeed, solving \eqref{function1} and \eqref{function2a2} will supply the main step for solving \eqref{section12}. On the other hand,
our method here to solve the more general problems \eqref{function1} and \eqref{function2a2} may have other applications.

We will write $(\varphi_1,\varphi_2)\leq(\tilde{\varphi}_1,\tilde{\varphi}_2)$ if $\varphi_i\leq\tilde{\varphi}_i,i=1,2.$
Let $\mathcal{R}\subset \mathbb{R}^2$ denote the rectangle
$$\mathcal{R}=[0, k_1]\times [0, k_2].$$
It is convenient to introduce the following notations:
\bess\begin{aligned}
&C_{\mathcal{R}}(\mathbb{R}^+,\mathbb{R}^2)=\{\varphi\in C(\mathbb{R}^+,\mathbb{R}^2)
:\varphi(s) \in \mathcal{R} \mbox{ for } s>0\},\\
&C_{\mathcal{R}}(\mathbb{R},\mathbb{R}^2)=\{\varphi\in C(\mathbb{R},\mathbb{R}^2)
: \varphi(s) \in \mathcal{R} \mbox{ for }s\in\mathbb{R}\}.
\end{aligned}\eess

\begin{definition}\lbl{s2} Suppose that $\overline{\varphi}=(\overline{\varphi}_1,\overline{\varphi}_2)\in C_{\mathcal{R}}(\mathbb{R},\mathbb{R}^2)$, $\underline{\varphi}=(\underline{\varphi}_1,\underline{\varphi}_2)\in C_{\mathcal{R}}(\mathbb{R},\mathbb{R}^2)$,
and $\overline{\varphi}_j,\underline{\varphi}_j$ are twice continuously differentiable in $\mathbb{R}\setminus \Omega_j$ for $j=1,2$, where $\Omega_1\subset [0,\infty)$ and $\Omega_2\subset \mathbb{R}$ are two finite sets, say $\Omega_1=\{\xi_i\geq 0: i=1,2,\cdots,m_1\}$ and $\Omega_2=\{\eta_i\in\mathbb{R}: i=1,2,\cdots,m_2\}$. Moreover,

{\rm \bf (i)} the functions $\overline{\varphi}$ and $\underline{\varphi}$ satisfy the inequalities
\bes\left\{\begin{aligned}
&d_1\overline{\varphi}_1''-c\overline{\varphi}_1'-\beta\overline{\varphi}_1+H_1(\overline{\varphi})\leq0\quad\mbox{for }
s\in\mathbb{R}^+\setminus \Omega_1,\\
&d_2\overline{\varphi}_2''-c\overline{\varphi}_2'-\beta\overline{\varphi}_2+H_2(\overline{\varphi})\leq0
\quad\mbox{for } s\in\mathbb{R}\setminus \Omega_2,\\
&\overline{\varphi}_1(s)=0\mbox{ for }s\leq0,\quad\overline{\varphi}_2(-\infty)=0
\end{aligned} \right.\lbl{upper}
\ees
and
\bes\left\{\begin{aligned}
&d_1\underline{\varphi}_1''-c\underline{\varphi}_1'-\beta\underline{\varphi}_1+H_1(\underline{\varphi})\geq0
\quad\mbox{for } s\in\mathbb{R}^+\setminus \Omega_1,\\
&d_2\underline{\varphi}_2''-c\underline{\varphi}_2'-\beta\underline{\varphi}_2+H_2(\underline{\varphi})\geq0
\quad\mbox{for } s\in\mathbb{R}\setminus \Omega_2,\\
&\underline{\varphi}_1(s)=0\mbox{ for }s\leq0,\quad\underline{\varphi}_2(-\infty)=0;
\end{aligned}\right.\lbl{lower}
\ees

{\rm \bf (ii)} the derivatives of $\overline{\varphi}$ and $\underline{\varphi}$  satisfy
\bes\begin{aligned}
&\underline{\varphi}_1'(\xi_i-)\leq\underline{\varphi}_1'(\xi_i+),\quad \overline{\varphi}_1'(\xi_i+)\leq\overline{\varphi}_1'(\xi_i-)\mbox{ for }\xi_i\in\Omega_1\setminus\{0\},\\
&\underline{\varphi}_2'(\eta_i-)\leq\underline{\varphi}_2'(\eta_i+),\quad \overline{\varphi}_2'(\eta_i+)\leq\overline{\varphi}_2'(\eta_i-)\mbox{ for }\eta_i\in\Omega_2.\lbl{lower1}
\end{aligned}\ees
Then $\overline{\varphi}$ and $\underline{\varphi}$ are called a weak \textbf{upper solution} and a weak \textbf{lower solution} of \eqref{function1}-\eqref{function2a2} associated with $\mathcal{R}$, respectively.
\end{definition}

\begin{prop}\lbl{s7} Assume that $(\textbf{A}_1)-(\textbf{A}_3)$ hold. Suppose \eqref{function1}-\eqref{function2a2} has a pair of upper and lower solutions associated with $\mathcal{R}$ satisfying
\[\mbox{
$\underline \varphi_1(s)\not\equiv 0$ \; and\;\;
$\sup_{t\leq s}\underline{\varphi}(t)\leq\overline{\varphi}(s)$  for $s\in\mathbb{R}$.}
\]
 Then \eqref{function1} has a monotone non-decreasing solution $\varphi$ satisfying  \eqref{function2a2}.
\end{prop}

Proposition \ref{s7} will play an important role in the proof of Theorem \ref{F4} in the next section. The proof of Proposition \ref{s7} is based on some upper and lower solution arguments and involves the Schauder fixed point theorem. Our proof is similar in spirit to that in several works on various different traveling wave problems (see, for example,
\cite{LLR,Ma,WNW,WZ}), but the detailed techniques are rather different. Since the proof is  long, it is postponed to Section 6
 at the end of the paper.

{\setlength{\baselineskip}{16pt}{\setlength\arraycolsep{2pt} \indent}}
\section{Semi-wave solutions}
\setcounter{equation}{0}

The aim of this section is to prove Theorem \ref{F4}. Firstly, we recall some known results for the Fisher-KPP equation
\bes\left\{\begin{aligned}
&d\chi''-c\chi'+a\chi(b-\chi)=0,\quad0<s<\infty,\\
&\chi(0)=0.
\end{aligned}\right.\lbl{KPP}
\ees
\begin{lem}\lbl{S22}
Let $a>0,b>0$ and $ d>0$ be fixed constants.
\begin{itemize}
\item[(i)] If $c\in[0,2\sqrt{abd})$, then  \eqref{KPP} has a unique  solution $\chi(s)$.

\item[(ii)] For each $c\in [0, 2\sqrt{abd})$, the solution $\chi(s)$ of \eqref{KPP} is strictly increasing and has the following asymptotic behavior
$$\chi(s)=b-[b_\chi +o(1)]e^{\frac{c-\sqrt{c^2+4ab}}{2d}s} \mbox{ as }s\rightarrow\infty,$$
where $b_\chi$ is a positive constant.
\end{itemize}
\end{lem}
The conclusion (i) can be found in \cite{DL10, BDK}, and the proof of (ii) is standard (see, for example, \cite{S}).

\begin{lem}\lbl{S21} {\rm (\cite{DWZ})}
Let $\tilde{f},\tilde{g}\in C([0,\infty))$ with $\tilde{g}$ nonnegative and not identically 0, and $c,d$ be given constants with $d>0$. Assume that $u_i(s)>0$ in $(0,\infty)$, $u_1(0)\leq u_2(0)$ and
$$cu_1'-du_1''-u_1[\tilde{f}(s)-\tilde{g}(s)u_1]\leq0\leq cu_2'-du_2''-u_2[\tilde{f}(s)-\tilde{g}(s)u_2],\quad 0<s<\infty.$$
If $\limsup_{s\rightarrow\infty}\frac{u_1(s)}{u_2(s)}\leq 1$, then
$$u_1(s)\leq u_2(s)\quad\mbox{ for }\; 0\leq s<\infty.$$
\end{lem}

Next, we consider the problem
\bes\left\{\begin{aligned}
&d\omega''+c\omega'+r(hu^\ast-\omega)(1-hu^\ast+\omega)=0,\quad0<s<\infty,\\
&\omega(0)=0.
\end{aligned}\right.\lbl{KP10}
\ees
\begin{lem}\lbl{S23}
For any constant $d>0$, $c\geq0$, the problem \eqref{KP10} admits a unique positive solution $\omega(s,c)$ satisfying $\omega(\infty, c)=h u^*$. Moreover, $\omega(s,c)$ is strictly increasing in $s$ for $s>0$, and there exists $b_\omega>0$ such that
 \bes
\omega(s)=hu^\ast-[b_\omega+o(1)]e^{\frac{-c-\sqrt{c^2+4dr}}{2d}s}\mbox{ as }s\rightarrow\infty.
\lbl{expk}
\ees

\end{lem}
\begin{proof}
It is easily seen that
\bes\underline{\omega}(s):=0,
\quad\overline{\omega}(s):=h u^*
\lbl{KP1}
\ees
are a pair of lower and upper solutions for \eqref{KP10}. Thus, \eqref{KP10} has at least one solution satisfying
$0\leq \omega(s)\leq hu^*$. The strong maximum principle infers that $\omega(s)<hu^*$ for $s>0$.

 We claim $\omega(s)$ is increasing and $\omega(\infty)=hu^\ast$. Rewrite \eqref{KP10} as $$-\left(de^{\frac{c}{d}s}\omega'\right)'=re^{\frac{c}{d}s}(hu^\ast-\omega)(1-hu^\ast+\omega).$$
Noting that $0\leq \omega<hu^\ast$, we have
$$-\left(de^{\frac{c}{d}s}\omega'\right)'>0\mbox{ for }s>0.$$
Hence, $e^{\frac{c}{d}s}\omega'$ is a decreasing function. We claim that $\omega(s)$ is monotone in $(R,\infty)$ for some
large $R>0$. Otherwise $\omega(s)$ is oscillating near $s=\infty$ and hence we can find a sequence $s_n\rightarrow\infty$ as $n\rightarrow\infty$ such that $\omega'(s_n)=0$. It follows that for any fixed $s>0$,
 $$e^{\frac{c}{d}s}\omega'(s)>\lim_{n\rightarrow\infty}e^{\frac{c}{d}s_n}\omega'(s_n)=0.$$
Thus, we have $\omega'>0$ in $(0,\infty)$, a contradiction to the assumption. Hence, for large $R>0$, $\omega$ is monotone in $(R,\infty)$.
By \eqref{KP10} and $0\leq \omega<hu^*$, we easily obtain $\omega(\infty)=hu^\ast$.
Furthermore, a simple calculation indicates that the ODE
system satisfied by $(\omega,\omega')$ has $(hu^\ast,0)$ as a saddle point. It follows from standard ODE theory that, there exists a constant $b_\omega>0$ such that \eqref{expk} holds, and
\[
\omega'(s)=\frac{c+\sqrt{c^2+4dr}}{2d}[b_\omega+o(1)]e^{\frac{-c-\sqrt{c^2+4dr}}{2d}s}=o(1) e^{-\frac{c}{d}s} \mbox{ as }s\rightarrow\infty.
\]
We thus obtain, for any fixed $s>0$,
 $$e^{\frac{c}{d}s}\omega'(s)>\lim_{s\rightarrow\infty}e^{\frac{c}{d}s}\omega'(s)=0.$$
Hence $\omega'>0$ in $(0,\infty)$ and $\omega$ is strictly increasing in $(0,\infty)$.

It remains to show the uniqueness of positive solution of \eqref{KP10} satisfying $\omega(\infty)=hu^*$.
Let $\omega_1$ and $\omega_2$ be two positive solutions of \eqref{KP10} satisfying $\omega_i(\infty)=hu^*$, $i=1,2$.
We easily see that $\omega_i\leq hu^*$ for otherwise there exists $s_i>0$ such that
\[
\omega_i(s_i)=\max_{s\geq 0}\omega_i(s)>hu^*, \; \omega_i'(s_i)=0\geq \omega_i''(s_i),
\]
which gives a contradiction to \eqref{KP10} when evaluated at $s=s_i$. The strong maximum principle then yields
$\omega_i(s)<h u^*$ for $s>0$. We may then argue as for $\omega$ above to obtain $\omega_i(\infty)=h u^*$.

Set $\chi_i(s):=1-hu^*+\omega_i(s)$ and we find that $\chi_i$ satisfies
\[
d\chi''-c\chi'+r(1-\chi)\chi=0 \mbox{ for } s>0, \; \chi(0)=1-hu^*,\; \chi(\infty)=1.
\]
By Lemma \ref{S21} we immediately obtain $\chi_1\equiv \chi_2$ and hence $\omega_1\equiv \omega_1$.
The uniqueness is thus proved.
\end{proof}

Now, we turn to consider system \eqref{section12}.
For convenience, we change it to a coorperative system by setting
\bess
\tilde{\phi}(s):=\phi(s),\quad \tilde{\psi}(s):=1-\psi(s).
\eess
Clearly $(\phi,\psi)$ solves \eqref{section12} if and only if $(\tilde{\phi},\tilde{\psi})$ satisfies
\bes\left\{
\begin{aligned}
&\tilde{\phi}''-c\tilde{\phi}'+\tilde{\phi}(1-k-\tilde{\phi}+k\tilde{\psi})=0,\quad\tilde{\phi}'>0\quad\mbox{for } s\in\mathbb{R}^+,\\
&d\tilde{\psi}''-c\tilde{\psi}'+r(1-\tilde{\psi})(h\tilde{\phi}-\tilde{\psi})=0,\quad\tilde{\psi}'>0\quad\mbox{for } s\in\mathbb{R},\\
&\tilde{\phi}(s)=0\mbox{ for }s\leq0,\quad \tilde{\psi}(-\infty)=0,\quad(\tilde{\phi}(\infty),\tilde{\psi}(\infty))=(u^\ast,hu^\ast).
\end{aligned}
\right.\lbl{sect1}\ees

\medskip

Let $\Sigma=\Sigma_1\times\Sigma_2$, where
\bess\begin{aligned}
&\Sigma_1=\{\tilde{\phi}\in C(\mathbb{R})\cap C^2([0,\infty)):\tilde{\phi}(s)\equiv0\, (s\leq0),\quad\tilde{\phi}'(s)>0\, (s>0),
\quad\tilde{\phi}(\infty)=u^\ast\},\\
&\Sigma_2=\{\tilde{\psi}\in C^2(\mathbb{R}):\tilde{\psi}(-\infty)=0,\quad\tilde{\psi}'(s)> 0\, (s\in\mathbb{R}),
\quad\tilde{\psi}(\infty)=hu^\ast\}.
\end{aligned}\eess
We define a functional $\theta$ on $\Sigma$ by
$$\theta(\tilde{\phi},\tilde{\psi})=\min\left\{\inf_{s\in\mathbb{R}^+}\Theta_1(\tilde{\phi},\tilde{\psi})(s),
\quad\inf_{s\in\mathbb{R}}\Theta_2(\tilde{\phi},\tilde{\psi})(s)\right\},$$
where
$$\begin{aligned}
&\Theta_1(\tilde{\phi},\tilde{\psi})(s)=\frac{\tilde{\phi}''(s)
+\tilde{\phi}(s)(1-k-\tilde{\phi}(s)+k\tilde{\psi}(s))}{\tilde{\phi}'(s)},\quad s\in\mathbb{R}^+,\\
&\Theta_2(\tilde{\phi},\tilde{\psi})(s)=\frac{d\tilde{\psi}''(s)
+r(1-\tilde{\psi}(s))(h\tilde{\phi}(s)-\tilde{\psi}(s))}{\tilde{\psi}'(s)},\quad s\in\mathbb{R}.
\end{aligned}$$
Clearly, if $(\tilde{\phi},\tilde{\psi})$ is a solution of \eqref{sect1} with $c\geq0$, then $(\tilde{\phi},\tilde{\psi})\in\Sigma$ and $\theta(\tilde{\phi},\tilde{\psi})=c$. Therefore,
$$c\leq c_0^\ast:=\sup_{(\tilde{\phi},\tilde{\psi})\in\Sigma}\theta(\tilde{\phi},\tilde{\psi}).$$

We will show $c_0^*=c_*>0$, where $c_*$ is given in Proposition \ref{F6}. For the moment we assume $c_0^*>0$ and
 prove that \eqref{sect1} has a solution for every $c\in[0,c_0^\ast)$ by using an upper and lower solution argument.
From Lemma \ref{S22}, the following equation
\bess\left\{\begin{array}{ll}\medskip
\displaystyle
\hat{\chi}''+\hat{\chi}(u^\ast-\hat{\chi})=0,\quad 0<s<\infty,\\
\hat{\chi}(0)=0
\end{array}\right.\lbl{}
\eess
has a unique strictly increasing solution $\hat{\chi}$ satisfying $\hat\chi(\infty)=u^\ast$.
We define
\bes\lbl{upfl}
\overline{\phi}(s)=\left\{\begin{array}{ll}\medskip
\displaystyle
0,&-\infty<s<0,\\
\hat{\chi}(s),&\quad 0\leq s<\infty,
\end{array}\right.
\quad\overline{\psi}(s)=\left\{\begin{array}{ll}\medskip
\displaystyle
hu^\ast-\omega(-s,0),&-\infty<s<0,\\
hu^\ast,&\quad 0\leq s<\infty.
\end{array}\right.
\ees

\begin{lem}\lbl{S24}
For $c\geq0$, the pair of functions $(\overline{\phi}(s),\overline{\psi}(s))$ is an upper solution of \eqref{sect1}
associated with $\mathcal{R}:=[0, u^*]\times [0, hu^*]$.
\end{lem}
\begin{proof} For $s\geq0$, we have
\bess\begin{array}{ll}\medskip
&\overline{\phi}''-c\overline{\phi}'+\overline{\phi}(1-k-\overline{\phi}+k\overline{\psi})\\
=&\hat{\chi}''-c\hat{\chi}'+\hat{\chi}(u^\ast-\hat{\chi})=-c\hat{\chi}'\leq0
\end{array}\eess
and
\bess\begin{array}{ll}\medskip
&d\overline{\psi}''-c\overline{\psi}'+r(1-\overline{\psi})(h\overline{\phi}-\overline{\psi})
=r h(1-hu^\ast)(\hat{\chi}-u^\ast)\leq0.
\end{array}\eess
For $s<0$, we have
\bess\begin{array}{ll}\medskip
&\overline{\phi}''-c\overline{\phi}'+\overline{\phi}(1-k-\overline{\phi}+k\overline{\psi})=0
\end{array}\eess
and
\bes\lbl{lsol}
\begin{array}{ll}\medskip
&d\overline{\psi}''-c\overline{\psi}'+r(1-\overline{\psi})(h\overline{\phi}-\overline{\psi})
=-c\overline{\psi}'\leq0.
\end{array}\ees
Moreover, it is easily seen that
\[
\overline{\psi}'(0-)\geq \overline{\psi}'(0+).
\]

Finally, by definition, $\overline \phi(s)\equiv 0$ for $s\leq 0$, $ \overline\phi(\infty)=u^*$,
$\overline \psi(-\infty)=0$ and $\overline\psi(\infty)=h u^*$. Hence $(\overline\phi,\overline\psi)$ meets all the requirements for an upper solution associated with $\mathcal{R}$ in Definition \ref{s2}.
This completes the proof.\end{proof}

\begin{lem}\lbl{S25} Assume $h, k\in (0,1)$, $c_0^*>0$ and $c\in[0,c_0^\ast)$. Then \eqref{sect1} has a solution $(\tilde{\phi},\tilde{\psi})$.
\end{lem}
\begin{proof}
For $c\in[0,c_0^\ast)$, by the definition of $c_0^\ast$, there exists $(\underline{\phi}(s),\underline{\psi}(s))\in\Sigma$ such that $\theta(\underline{\phi},\underline{\psi})>c$. Thus we have
\bess\left\{\begin{aligned}
&\underline{\phi}''-c\underline{\phi}'+\underline{\phi}(1-k-\underline{\phi}+k\underline{\psi})\geq0,
\quad\underline{\phi}'>0,\quad s\in\mathbb{R}^+,\\
&d\underline{\psi}''-c\underline{\psi}'+r(1-\underline{\psi})(h\underline{\phi}-\underline{\psi})\geq0,
\quad\underline{\psi}'>0,\quad s\in\mathbb{R},\\
&\underline{\phi}(s)=0\mbox{ for }s\leq0,\quad \underline{\psi}(-\infty)=0,\quad(\underline{\phi}(\infty),\underline{\psi}(\infty))=(u^\ast,hu^\ast).
\end{aligned}\right.\eess
Hence, $(\underline{\phi},\underline{\psi})$ is a lower solution of \eqref{sect1} associated with $\mathcal{R}:=[0, u^*]\times [0, hu^*]$.

Next, we show that $(\underline{\phi}(s),\underline{\psi}(s))\leq(\overline{\phi}(s),\overline{\psi}(s))$ for $s\in\mathbb{R}$, where
$(\overline{\phi},\overline{\psi})$ is the upper solution obtained in Lemma \ref{S24}.
Clearly, $\underline{\psi}(s)\leq\overline{\psi}(s)$ for $s>0$ and $\underline{\phi}(s)=\overline{\phi}(s)=0$ for $s<0$.
We only need to show that $\underline{\psi}(s)\leq\overline{\psi}(s)$ for $s\leq0$, and $\underline{\phi}(s)\leq\overline{\phi}(s)$ for $s\geq0$.
Let $\underline{\psi}_1(s)=1-\underline{\psi}(-s)$ and $\overline{\psi}_1(s)=1-\overline{\psi}(-s)$. In view of
$\underline{\phi}(s)=\overline \phi(s)=0$ for $s<0$ and \eqref{lsol},
we have
\bess\begin{aligned}&-d\underline{\psi}_1''-c\underline{\psi}_1'-r\underline{\psi}_1(1
-\underline{\psi}_1)\geq 0\geq
-d\overline{\psi}_1''-c\overline{\psi}_1'-r\overline{\psi}_1(1
-\overline{\psi}_1)  \mbox{ for } s>0,\\
&\underline{\psi}_1(\infty)=\overline{\psi}_1(\infty)=1,\  \ \underline{\psi}_1(0)>\overline{\psi}_1(0).
\end{aligned}\eess
By Lemma \ref{S21}, $\underline{\psi}_1(s)\geq\overline{\psi}_1(s)$ for $s\geq0$.
Hence, $\underline{\psi}(s)\leq\overline{\psi}(s)$ for $s\leq0$. Similarly, we can prove $\underline{\phi}(s)\leq\overline{\phi}(s)$ for $s\geq0$.  Thus
\[
\mbox{$(\underline{\phi}(s),\underline{\psi}(s))\leq(\overline{\phi}(s),\overline{\psi}(s))$ for $s\in\mathbb{R}$.}
\]
The monotonicity of $\underline{\phi}(s)$ and $\underline{\psi}(s)$ then infers that
\[
\sup_{t\leq s}(\underline{\phi}(t),\underline{\psi}(t))\leq (\overline{\phi}(s),\overline{\psi}(s)) \mbox{ for } s\in\mathbb{R}.
\]
Therefore we can apply Proposition \ref{s7} to conclude that \eqref{sect1} has a positive solution $(\tilde{\phi},\tilde{\psi})$ for each $c\in[0,c_0^\ast)$, except that we only have $\tilde \phi'(s)\geq 0$ for $s>0$ and $\tilde\psi'(s)\geq 0$ for $s\in\mathbb R$.

It remains to prove $\tilde{\phi}'(s)>0$ for $s>0$ and $\tilde{\psi}'(s)>0$  for $s\in\mathbb{R}$.
Since $\tilde{\phi}'(s)\geq0$ for $s\in\mathbb{R}\setminus\{0\}$ and $\tilde{\psi}'(s)\geq0$  for $s\in\mathbb{R}$, and none of them is identically 0, applying the strong maximum principle to the cooperative system satisfied by $(\phi',\psi')$, we have $\tilde{\phi}'(s)>0$ for $s>0$ and $\tilde{\psi}'(s)>0$  for $s\in\mathbb{R}$.
\end{proof}

To prove uniqueness for solutions of \eqref{sect1}, we need to investigate the asymptotic behavior of solutions to \eqref{sect1} as $s\rightarrow\infty$. To this end we consider the linearized equation of \eqref{sect1} at $(u^\ast,hu^\ast)$:
\bes\left\{\begin{array}{ll}\medskip
 \check{\phi}''-c\check{\phi}'-u^\ast\check{\phi}+ku^\ast\check{\psi}=0,\\
d \check{\psi}''-c\check{\psi}'-rv^\ast\check{\psi}+rhv^\ast\check{\phi}=0.
\end{array}\right.\lbl{f7}
\ees
If $(\check{\phi},\check{\psi})=(me^{\mu s},ne^{\mu s})$ solves \eqref{f7},
then $(m,n)$ and $\mu$ must satisfy
\bes
A(\mu)(m,n)^T=(0,0)^T,
\lbl{eigenvector}\ees
where
\bess
A(\mu)=\left(
\begin{array}{cc}\medskip
 \mu^2-c\mu-u^\ast&ku^\ast\\
hrv^\ast&d\mu^2-c\mu-rv^\ast
\end{array}\right).
\eess
Let
\bess\lbl{poly}
P_1(\mu):={\rm det} (A(\mu))=(\mu^2-c\mu-u^\ast)(d\mu^2-c\mu-rv^\ast)-khru^\ast v^\ast.
\eess
Then \eqref{eigenvector} has a nonzero solution $(m,n)^T$ if and only if $P_1(\mu)=0$.

Let
$$\mu_1^\pm:=\frac{c\pm\sqrt{c^2+4u^\ast}}{2}\mbox{ and } \mu_2^\pm:=\frac{c\pm\sqrt{c^2+4drv^\ast}}{2d}$$
be the two roots of
$$\mu^2-c\mu-u^\ast=0\mbox{ and }d\mu^2-c\mu-rv^\ast=0,$$
respectively.
Clearly
$$
P_1(0)=(1-kh)ru^\ast v^\ast>0,\; P_1(\pm\infty)=\infty\mbox{ and } P_1(\mu_i^\pm)=-khru^*v^*<0\mbox{ for } i=1,2.
$$
Hence, for any $c\geq0$, $P_1(\mu)=0$ has four different real roots $\hat{\mu}_i\,(i=1,2,3,4)$ satisfying
\bes\begin{aligned}
&\hat{\mu}_1<\min\{\mu_1^-,\mu_2^-\}\leq\max\{\mu_1^-,\mu_2^-\}<\hat{\mu}_2<0,\\
&0<\hat{\mu}_3<\min\{\mu_1^+,\mu_2^+\}\leq\max\{\mu_1^+,\mu_2^+\}<\hat{\mu}_4.\end{aligned}\lbl{exp1}\ees

\begin{lem}\lbl{S26}Let $(\tilde{\phi}(s),\tilde{\psi}(s))$ be a solution of \eqref{sect1}. Then there exist positive constants $m$ and $n$ independent of $(\tilde{\phi},\tilde{\psi})$, and a positive constant $\beta$ depending on $(\tilde{\phi},\tilde{\psi})$, such that
\bes
(\tilde{\phi}(s),\tilde{\psi}(s))=(u^\ast,hu^\ast)- \beta e^{\hat{\mu}_2s}(m,n)[1+o(1)]\mbox{ as }s\rightarrow\infty.\lbl{expansion1}\ees
\end{lem}
\begin{proof}\,
Let $(\tilde{\phi}(s),\tilde{\psi}(s))$ be an arbitrary solution of \eqref{sect1}. The inequalities \eqref{exp1} imply that the first
order ODE system satisfied by $(\tilde{\phi}(s),\tilde{\phi}'(s),\tilde{\psi}(s),\tilde{\psi}'(s))$ has a critical point at $(u^\ast,0,hu^\ast,0)$, which is a saddle point. By standard stable manifold theory (see, e.g., Theorem 4.1 and its proof in Chapter 13 of \cite{Co}), we can conclude that
$$(u^\ast,hu^\ast)-(\tilde{\phi}(s),\tilde{\psi}(s))\rightarrow(0,0)\mbox{ exponentially as }s\rightarrow\infty.$$
Let $(\hat{\phi},\hat{\psi})=(u^\ast,hu^\ast)-(\tilde{\phi}(s),\tilde{\psi}(s))$.  Then $(\hat{\phi},\hat{\psi})$ satisfies
\bes\left\{\begin{aligned}
& \hat{\phi}''-c\hat{\phi}'-u^\ast\hat{\phi}+ku^\ast\hat{\psi}
 +\delta_1(s)\hat{\phi}+\delta_2(s)\hat{\psi}=0,\\
&d \hat{\psi}''-c\tilde{\psi}'-rv^\ast\hat{\psi}+hrv^\ast\hat{\phi}
+\delta_3(s)\hat{\phi}+\delta_4(s)\hat{\psi}=0,
\end{aligned}\right.\lbl{f6}
\ees
where
$$\delta_1(s):=\hat{\phi}(s),\delta_2(s):=-k\hat{\phi}(s),\delta_3(s):=-rh\hat{\psi}(s),\delta_4(s):=r\hat{\psi}(s).$$
Clearly,
\bess
\delta_i(s)\rightarrow0\mbox{ exponentially as }s\rightarrow\infty\mbox{ for }i=1,2,3,4.
\eess

Now we turn to consider the linear system \eqref{f7}.
Recall that $P_1(\mu)=0$ has four different real roots satisfying $\hat{\mu}_1<\hat{\mu}_2<0<\hat{\mu}_3<\hat{\mu}_4$.
Let $(m_i,n_i)$ be an eigenvector corresponding to $\mu=\hat{\mu}_i$ in \eqref{eigenvector}, i.e.,
\[
(m_i, n_i)\not=(0,0) \mbox{ and }  A(\hat \mu_i)(m_i,n_i)^T=(0,0)^T.
\]
 Then \eqref{f7} has four linearly independent solutions
\[
 \Upsilon_i=(m_i,n_i)e^{\hat{\mu}_is},\; i=1,2,3,4,
\]
which form a fundamental system for \eqref{f7}.
Applying Theorem 8.1 in Chapter 3 of \cite{Co} to the system \eqref{f6}, viewed as a perturbed linear system of \eqref{f7}, we conclude that \eqref{f6} has four linearly independent solutions $\tilde{\Upsilon}_i,\; i=1,2,3,4,$ satisfying
$$
\tilde{\Upsilon}_i(s)=(1+o(1))\Upsilon_i(s)\mbox{ as }s\rightarrow\infty,\; i=1,2,3,4,
$$
which form a fundamental system for \eqref{f6} (viewed as a linear system).
So the solution $(\hat{\phi},\hat{\psi})$ of \eqref{f6} can be represented as
$$(\hat{\phi}(s),\hat{\psi}(s))=\sum_{i=1}^4\beta_i\tilde{\Upsilon}_i(s),
$$
where $\beta_i\; (i=1, 2, 3, 4)$ are constants.

Since $(\hat{\phi}(\infty),\hat{\psi}(\infty))=(0,0)$, and $0<\hat \mu_3<\hat\mu_4$, we necessarily have  $\beta_3=\beta_4=0$.
We claim that $\beta_2\not=0$. Otherwise we necessarily have $\beta_1\not=0$ and
\[
(\hat{\phi}(s),\hat{\psi}(s))=\beta_1\tilde{\Upsilon}_1(s)=(1+o(1))\beta_1(m_1,n_1)e^{\hat{\mu}_1s}\mbox{ as } s\rightarrow\infty.
\]
However, it is easily checked that all the four elements of the matrix $A(\hat\mu_1)$ are positive, which implies that
$m_1\cdot n_1<0$ and so the two components of the vector $\beta_1(m_1, n_1)$ have opposite signs. It follows that
for all large $s$, $(\hat{\phi}(s),\hat{\psi}(s))$ has a component which is negative, contradicting the fact that $(\hat{\phi}(s),\hat{\psi}(s))>(0,0)$ for all $s>0$. Therefore we must have $\beta_2\not=0$. It is also easily checked that the two rows of the matrix
$A(\hat \mu_2)$ have opposite signs and so $m_2\cdot n_2>0$. For definiteness, we may assume that $m_2$ and $n_2$ are positive. Moreover, due to $\hat \mu_1<\hat \mu_2<0$, we have
\[
(\hat{\phi}(s),\hat{\psi}(s))=\sum_{i=1}^2\beta_i\tilde{\Upsilon}_i(s)=(1+o(1))\beta_2(m_2,n_2)e^{\hat{\mu}_2s}\mbox{ as } s\rightarrow\infty.
\]
Using $(\hat{\phi}(s),\hat{\psi}(s))>(0,0)$ for all $s>0$ we further obtain that $\beta_2>0$, and hence
\eqref{expansion1} holds with $(m,n):=(m_2,n_2)$ and $\beta:=\beta_2$.
\end{proof}

\begin{lem}\lbl{S27}The  solution of \eqref{sect1} is unique.
\end{lem}
\begin{proof}\,
 Let $(\phi,\psi)$  and $(\phi_1,\psi_1)$ be two arbitrary solutions of \eqref{sect1}. We are going to show that
\bes\lbl{geq}
(\phi(s),\psi(s))\geq (\phi_1(s),\psi_1(s)) \mbox{ for } s\in\mathbb R.
\ees
Note that if we are able to prove \eqref{geq}, then the same argument can also be used to show $(\phi_1,\psi_1)\geq (\phi,\psi)$. Hence uniqueness will follow if we can show
\eqref{geq}.

For $s\in\mathbb R$ and $\xi\geq 0$, define
\bess
\underline{\phi}(s)=\underline \phi^\xi(s):=\phi_1(s-\xi),
\quad\underline{\psi}(s)=\underline{\psi}^\xi(s):=\psi_1(s-\xi).
\eess
We claim that there exists a constant $\xi_0>0$ such that, for every $\xi\geq\xi_0$
\bes\lbl{xi_0}
(\underline{\phi}^\xi(s),\underline{\psi}^\xi(s))\leq(\phi(s),\psi(s))\mbox{ for all }s\in\mathbb{R}.
\ees

Since $\psi_1(-\infty)=0<\psi(0)$, there exists $\xi_1>0$ large enough such that
$\psi_1(-\xi_1)\leq\psi(0)$. Then
$\underline{\psi}^\xi(0)=\psi_1(-\xi)\leq\psi(0)$ for all $\xi\geq\xi_1$, and $\psi(s)$, $\underline{\psi}^\xi(s)$ satisfy
\bess
\begin{aligned}
&d\psi''-c\psi'+r\psi(\psi-1)=0= d\underline{\psi}''-c\underline{\psi}'+r\underline{\psi}(\underline{\psi}-1),\quad s\leq0,\\
&\psi(-\infty)=\underline{\psi}(-\infty)=0,\quad \psi(0)\geq\underline{\psi}(0).
\end{aligned}
\eess
Let $u_1(s):=1-\underline{\psi}(-s)$ and $u_2(s):=1-\psi(-s)$. Then $u_1$ and $u_2$ satisfy
\bess\left\{
\begin{aligned}
&du_2''+cu_2'+ru_2(1-u_2)=0=du_1''+cu_1'+ru_1(1-u_1),\quad s\geq0,\\
&u_2(\infty)=u_1(\infty)=1,\quad u_2(0)\leq u_1(0).
\end{aligned}\right.
\eess
By Lemma \ref{S21}, we deduce that $u_2(s)\leq u_1(s)$  for all $s\geq0$. We thus obtain
\bes\lbl{s<0}
 \mbox{ $\underline{\psi}^\xi(s)\leq\psi(s)$ for all $s\leq0$ and $\xi\geq \xi_1$.}
\ees

Applying Lemma \ref{S26},  we can find $\xi_2>\xi_1$ and $s_0\gg1$ such that
$$
\underline{\phi}^{\xi_2}(s) \leq \phi(s) \mbox{ for all } s\geq s_0.
$$
Denote $\xi_0=\xi_2+s_0$. Since $\phi_1$ is nondecreasing in $\mathbb{R}$ and is identically 0 in $(-\infty, 0]$, it follows that
\bes\lbl{R}
\underline{\phi}^\xi(s)\leq\phi(s) \mbox{ for } s\in\mathbb{R},\; \xi\geq\xi_0.
\ees
Similarly, for $s>0$, $\psi(s)$ and $\underline{\psi}^\xi(s)$ satisfy
\bess\left\{
\begin{aligned}
&-d\psi''+c\psi'=r(1-\psi)
(h\phi-\psi),\quad s>0.\\
&-d\underline{\psi}''+c\underline{\psi}'=r(1-\underline{\psi})
(h\underline{\phi}^\xi-\underline{\psi})\leq r(1-\underline{\psi})
(h\phi-\underline{\psi}),\quad s>0,\\
&\psi(\infty)=\underline{\psi}(\infty)=hu^\ast,\quad \psi(0)\geq\underline{\psi}(0).
\end{aligned}
\right.
\eess
Let $1-\underline{\psi}=v_1$ and $1-\psi_1=v_2$.
Then $v_1$ and $v_2$ satisfy
\bess\left\{
\begin{aligned}
&dv_2''-cv_2'=-rv_2
(1-h\phi-v_2),\quad s>0.\\
&dv_1''-cv_1'\leq r(1-\underline{\psi})
(h\phi-\underline{\psi})= -rv_1
(1-h\phi-v_1),\quad s>0,\\
&v_2(0)\leq v_1(0),\quad v_2(\infty)=v_1(\infty)=1-hu^\ast.
\end{aligned}\right.
\eess
Using Lemma \ref{S21} again,  we have $v_2(s)\leq v_1(s)$ for all $s\geq0$ and $\xi\geq\xi_0$, and hence
\bes\lbl{s>0}
\mbox{ $\psi(s)\geq\underline{\psi}^\xi(s)$ for all $s>0$ and $\xi\geq\xi_0$.}
\ees
Combining \eqref{s>0}, \eqref{s<0} and \eqref{R}, we immediately obtain \eqref{xi_0}.

Define
\bess
\bar{\zeta}:=\inf\{\zeta_0>0:(\phi(s), \psi(s))\geq(\phi_1(s-\zeta), \psi_1(s-\zeta) \mbox{ for } s\in\mathbb{R},\quad\forall \zeta\geq \zeta_0\}.
\eess
By \eqref{xi_0}, $\bar\zeta$ is well defined. Since $\phi(0)=0<\phi_1(-\zeta)$ for $\zeta>0$, we have $\bar\zeta\geq 0$.
Clearly,
$$
(\phi(s), \psi(s))\geq(\phi_1(s-\bar{\zeta}),\psi_1(s-\bar{\zeta})) \mbox{ for } s\in\mathbb{R}.
$$

If $\bar{\zeta}=0$, then the above inequality already yields \eqref{geq}, and the proof is finished.
  Suppose $\bar{\zeta}>0$. We are going to derive a contradiction. To simplify notations we write
\[
(\phi_{\bar\zeta}(s),\psi_{\bar\zeta}(s))=(\phi_1(s-\bar \zeta), \psi_1(s-\bar\zeta)),
\]
 and set
\[
 P(s):=\phi(s)-\phi_{\bar{\zeta}}(s),\;\;   Q(s):=\psi(s)-\psi_{\bar{\zeta}}(s).
\]
Then the nonnegative functions $ P$ and $ Q$ satisfy
\bes\left\{\begin{aligned}
& P''-c P'+(1-k-\phi-\phi_{\bar{\zeta}}+k\psi_{\bar{\zeta}}) P+k\phi Q=0,&s>\bar{\zeta},\\
&d Q''-c Q'+r(\psi+\psi_{\bar{\zeta}}-1-h\psi) Q+rh(1-\psi) P=0,&s\in\mathbb{R},\\
& P(0)= P(\infty)= Q(-\infty)= Q(\infty)=0.
\end{aligned}\right.\lbl{ff10}
\ees
The strong maximum principle implies that $ P(s)>0$ for $s\geq\bar\zeta$ and $ Q(s)>0$ for $s\in\mathbb{R}$. Rewrite \eqref{ff10} as
\bess\left\{\begin{aligned}
& P''-c P'-u^\ast P+ku^\ast Q+\epsilon_1(s) P+\epsilon_2(s) Q=0,&s>\bar\zeta,\\
&d Q''-c Q'-rv^\ast Q+rhv^\ast P+\epsilon_3(s) P+\epsilon_4(s) Q=0,&s\in\mathbb{R},
\end{aligned}\right.\lbl{f10}
\eess
where
\bess\begin{aligned}
&\epsilon_1=1-k-\phi-\phi_{\bar{\zeta}}+k\psi_{\bar{\zeta}}+u^\ast,\quad\epsilon_2=k\phi-ku^\ast,\\
&\epsilon_3=rh(1-\psi)-rhv^\ast,\quad\epsilon_4=r(\psi+\psi_{\bar{\zeta}}-1-h\psi)+rv^\ast.
\end{aligned}\eess
By Lemma \ref{S26}, $\epsilon_i(s)\rightarrow0$ exponentially as $s\rightarrow\infty$ for $i=1,2,3,4$. We may now
repeat the proof process of Lemma \ref{S26} to obtain
$$( P(s), Q(s))=(\bar{C}_1+o(1),\bar{C}_2+o(1))e^{\hat{\mu}_2s}\mbox{ as }s\rightarrow\infty,$$
where $\bar{C}_1,\; \bar C_2$ are positive constants. By Lemma \ref{S26}, there are positive constants $C_*$ and $C$ such that
\bess\begin{aligned}
&(\phi(s),\psi(s))=(u^\ast,hu^\ast)-C_\ast(m+o(1),n+o(1))e^{\hat{\mu}_2s}\mbox{ as }s\rightarrow\infty,\\
&(\phi_{\bar\zeta}(s),\psi_{\bar\zeta}(s))=(u^\ast,hu^\ast)-C(m+o(1),n+o(1))e^{\hat{\mu}_2s}\mbox{ as }s\rightarrow\infty,
\end{aligned}\eess
which lead to
$$\left(m(Ce^{-\hat{\mu}_2\bar{\zeta}}-C_\ast),
n(Ce^{-\hat{\mu}_2\bar{\zeta}}-C_{\ast})\right)=(\bar{C}_1,\bar{C}_2)>(0,0).$$
Therefore, there exists $\epsilon_0>0$ sufficiently small so that for any $\epsilon\in(0,2\epsilon_0]$,
$$\left(m(Ce^{-\hat{\mu}_2(\bar{\zeta}-\epsilon)}-C_{\ast}),
n(Ce^{-\hat{\mu}_2(\bar{\zeta}-\epsilon)}-C_{\ast})\right)>(0,0).$$
It follows that, for all large $s$, say $s\geq M>\bar\zeta$, we have
$$(\phi(s),\psi(s))\geq (\phi_1(s-\bar{\zeta}+\epsilon), \psi_1(s-\bar{\zeta}+\epsilon)) \; (\forall\epsilon\in(0,\epsilon_0]).$$
Since $( P(s), Q(s))>(0,0)$ for $s\in[\bar\zeta,M],$
by continuity, we can find ${\epsilon_1}\in(0,\epsilon_0]$ such that, for every $\epsilon\in (0,\epsilon_1]$,
\[
(\phi(s),\psi(s))\geq (\phi_1(s-\bar{\zeta}+{\epsilon}), \psi_1(s-\bar{\zeta}+{\epsilon}))\mbox{ for } s\in[ \bar\zeta, M].
\]
Hence
\bes\lbl{phi-psi-half}
(\phi(s),\psi(s))\geq (\phi_1(s-\bar{\zeta}+{\epsilon}), \psi_1(s-\bar{\zeta}+{\epsilon}))\mbox{ for } s\geq \bar\zeta,\; \epsilon\in (0,\epsilon_1].
\ees

Since $\phi(\bar\zeta)>0=\phi_1(0)$, by continuity, there exists $\epsilon_2\in (0,\epsilon_1]$ such that
$\phi(\bar\zeta-\epsilon)\geq \phi_1(\epsilon)$ for $\epsilon\in (0,\epsilon_2]$. It follows that
\[
\phi(s)\geq \phi_1(s-\bar\zeta+\epsilon) \mbox{ for } s\in [\bar\zeta-\epsilon,\bar\zeta].
\]
Since $\phi_1(s-\bar\zeta+\epsilon)\equiv 0 $ for $s\leq \bar\zeta-\epsilon$, the above inequality holds for all $s\in(-\infty,\bar\zeta]$.
This and \eqref{phi-psi-half} imply
\bes\lbl{phi-R}
\phi(s)\geq \phi_1(s-\bar\zeta+\epsilon) \mbox{ for } s\in \mathbb R,\;\epsilon\in (0,\epsilon_2].
\ees

Denote
\[
(\phi_\epsilon(s),\psi_\epsilon(s)):=(\phi_1(s-\bar\zeta+\epsilon),\psi_1(s-\bar\zeta+\epsilon)).
\]
We obtain, for any fixed $\epsilon\in (0,\epsilon_2]$,
\[
-d\psi_\epsilon''+c\psi'_\epsilon=r(1-\psi_\epsilon)(h\phi_\epsilon-\psi_\epsilon)\leq r(1-\psi_\epsilon)(h\phi-\psi_\epsilon)
\mbox{ for } s\in \mathbb R.
\]
Moreover, $\psi_\epsilon(-\infty)=0,\; \psi_\epsilon(\infty)=hu^*$.
Hence $u_2(s):=1-\psi_\epsilon(-s)$ satisfies
\[
d u_2''+cu_2'+r u_2(1-h\phi(-s)-u_2)\leq 0 \mbox{ for } s\in\mathbb R,\; u_2(\infty)=1.
\]
Since $u_1(s):=1-\psi(-s)$ satisfies
\[
d u_1''+cu_1'+r u_1(1-h\phi(-s)-u_1)= 0 \mbox{ for } s\in\mathbb R,\; u_1(\infty)=1
\]
and $u_1(-\bar \zeta)\leq u_2(-\bar\zeta)$, we may apply Lemma \ref{S21} to conclude that
$u_1(s)\leq u_2(s)$ for $s\geq -\bar\zeta$, i.e., $\psi(s)\geq \psi_\epsilon(s)$ for $s\leq\bar\zeta$.
Combining this with \eqref{phi-psi-half}, we obtain
\[
\psi(s)\geq \psi_\epsilon(s)=\psi_1(s-\bar\zeta+\epsilon) \mbox{ for } s\in\mathbb R,\;\epsilon\in (0,\epsilon_2].
\]
This and \eqref{phi-R} clearly contradict the definition of $\bar\zeta$.
Hence the  case $\bar{\zeta}>0$ can not happen, and the proof is complete.
\end{proof}

Next we will make use of problem \eqref{entirewave0}. Setting $\tilde{\Phi}(s):=\Phi(s),\tilde{\Psi}(s):=1-\Psi(s)$, we may change
 \eqref{entirewave0} to the following cooperative system
\bes\left\{
\begin{aligned}
&\tilde{\Phi}''-c\tilde{\Phi}'+\tilde{\Phi}(1-k-\tilde{\Phi}+k\tilde{\Psi})=0,\; \tilde \Phi'>0, \quad s\in\mathbb{R},\\
&d\tilde{\Psi}''-c\tilde{\Psi}'+r(1-\tilde{\Psi})(h\tilde{\Phi}-\tilde{\Psi})=0,\; \tilde \Psi'>0, \quad s\in\mathbb{R},\\
&(\tilde{\Phi},\tilde{\Psi})(-\infty)=(0,0),\quad(\tilde{\Phi},\tilde{\Psi})(\infty)=(u^\ast,hu^\ast ).
\end{aligned}
\right.\lbl{sect1entire}\ees

From Proposition \ref{F6}, we know that there exists $c_*\geq2\sqrt{1-k}$ such that  \eqref{sect1entire}  possesses a solution if and only if $c\geq c_*$.

In what follows, we shall show $c_0^\ast=c_{\ast}$ and \eqref{sect1} has no solution for $c\geq c_{\ast}$.

\begin{lem}\lbl{aa1} $c_0^\ast\geq c_{\ast}$.
\end{lem}
\begin{proof}\,
Let $(\Phi_0,\Psi_0)$ be a solution of \eqref{sect1entire} with $c=c_{\ast}$. It is easily checked that $(0,0,0,0)$ is a saddle equilibrium point of the ODE system satisfied by $(\Phi_0,\Phi_0',\Psi_0,\Psi'_0)$. It follows that
\[
\Phi_0,\Phi_0',\Psi_0,\Psi'_0\to 0 \mbox{ exponentially  as $s\to-\infty$}.
\]

Following the idea in the proof of Lemma \ref{S26}, we rewrite the equation satisfied by $\Phi_0$ as
\[
\Phi_0''-c_*\Phi_0'+(1-k)\Phi_0+\epsilon(s)\Phi_0=0
\]
with
\[\epsilon(s):=k\Psi_0(s)-\Phi_0(s)\to 0 \mbox{ exponentially as $s\to-\infty$},
\]
and view it as a perturbed linear equation to
\[
\Phi''-c_*\Phi'+(1-k)\Phi=0.
\]
Using the fundamental solutions of this latter equation we see that, as
$s\rightarrow-\infty$, the asymptotic behaviour of $(\Phi_0, \Phi_0')$ is given by
\bes\lbl{asymp1}
(\Phi_0(s), \Phi_0'(s))=\left\{\begin{array}{ll}\medskip
\displaystyle
(1, \alpha_1)k_0 e^{\alpha_1 s}(1+o(1)), & \mbox{ when } c_*>2\sqrt{1-k},\\
(1,\alpha_1)k_0|s|e^{\alpha_1 s}(1+o(1))&\\
\hspace{.2cm} \mbox{ or } (1,\alpha_1)k_0e^{\alpha_1 s}(1+o(1)),& \mbox{ when } c_*=2\sqrt{1-k}
\end{array}\right.
\ees
for some $k_0>0$, $\alpha_1\in\left\{\frac 12\left(c_{\ast}+\sqrt{c_{\ast}^2-4(1-k)}\right), \frac 12\left(c_{\ast}-\sqrt{c_{\ast}^2-4(1-k)}\right)\right\}$.

Fix $\epsilon\in(0, k/\alpha_1)$ small. In view of \eqref{asymp1}, there exists a constant $M_0<0$ such that
\bes\lbl{qq1}\frac{\Phi_0'(s)}{\Phi_0(s)}\geq\frac{\alpha_1}{2},\quad\max\{\Phi_0(s),
\Psi_0(s)\}<\min\left\{\frac{\epsilon\alpha_1}{4k},1-k\right\}\mbox{ for }s<M_0.\ees

Next, we prove that system \eqref{sect1} has a solution for $c=c_{\ast}-\epsilon$. To this end, we will treat the cases $c_{\ast}>2\sqrt{1-k}$ and $c_{\ast}=2\sqrt{1-k}$ separately.

{\bf Case 1}: $c_{\ast}>2\sqrt{1-k}$.

Introduce an auxiliary function
$$p_1(s)=0\mbox{ for }s\geq M_0,\quad p_1(s)=e^{\beta_1 s}\mbox{ for }s\leq M_0-1$$
and for $s\in(M_0-1,M_0)$, $p_1(s)>0$, $p_1'(s)\leq0$, where $\beta_1=\left(c_{\ast}-\sqrt{c_{\ast}^2+2dr}\right)/(2d)<0$
and  $M_0$ is given by \eqref{qq1}. Moreover, $p_1(s)$ is $C^2$ everywhere.  Define
\[
\psi_1(s):=\Psi_0(s)-\epsilon_1p_1(s),
\]
where the positive constant $\epsilon_1$ will be determined later.

We now calculate
\bes\begin{aligned}\lbl{up1}
&d\psi_1''-(c_{\ast}-\epsilon)\psi_1'+r(1-\psi_1)(h\Phi_0-\psi_1)\\
=&r(1-\psi_1)(h\Phi_0-\psi_1)-r(1-\Psi_0)(h\Phi_0-\Psi_0)-d\epsilon_1p_1''-\epsilon\epsilon_1 p_1'+\epsilon\Psi_0'+c_{\ast}\epsilon_1p_1'\\
=&\epsilon\Psi_0'+\epsilon_1[rp_1(1+h\Phi_0-2\Psi_0+\epsilon_1p_1)-dp_1''-\epsilon p_1'+c_{\ast}p_1'].
\end{aligned}\ees
Hence we can fix $\epsilon_1>0$ sufficiently small so that, for $s\in[M_0-1,M_0]$,
\bess\begin{aligned}\lbl{up2}d\psi_1''-(c_{\ast}-\epsilon)\psi_1'
+r(1-\psi_1)(h\Phi_0-\psi_1)>0\end{aligned}\eess
and
$$\psi_1'(s)>0,\quad\psi_1(s)>0.$$
By the definition of $p_1(s)$ for $s\leq M_0-1$, clearly $\psi_1'(s)>0$  for $s\leq M_0-1$, and $\psi_1(s)\rightarrow-\infty$ as $s\rightarrow-\infty$.
Hence there exists a unique constant $M_1<M_0-1$ such that $\psi_1(M_1)=0$. We define
\bess
\underline{\Phi}(s)=\Phi_0(s)\mbox{ for }s\geq M_1,
\quad\underline{\Psi}(s)=\left\{\begin{array}{ll}\medskip
\displaystyle
\Psi_0(s)-\epsilon_1p_1(s),&s>M_1,\\
0,&s\leq M_1.
\end{array}\right.
\eess
For $s\in(M_1,M_0-1)$, by the choice of $\Psi_0\leq1/4$ in \eqref{qq1} for $s\leq M_0$, we have
\bess\begin{aligned}
&rp_1(1+h\Phi_0-2\Psi_0+\epsilon_1p_1)-dp_1''-\epsilon p_1'+c_{\ast}p_1'\\
>&rp_1(1+h\Phi_0-2\Psi_0)-dp_1''+c_{\ast}p_1'\\
>&-dp_1''+c_{\ast}p_1'+\frac{r}{2}p_1=0.
\end{aligned}\eess
Therefore, it follows from \eqref{up1} that
$$d\underline{\Psi}''-(c_{\ast}-\epsilon)\underline{\Psi}'+r(1-\underline{\Psi})(h\underline{\Phi}-\underline{\Psi})>0,\quad s\in(M_1,M_0-1).$$
Since $(\underline{\Phi},\underline{\Psi})=(\Phi_0,\Psi_0)$ for $s>M_0$, and $\underline{\Psi}(s)=0$ for $s\leq M_1$, it is easy to verify that
for any smooth extension of $\underline\Phi(s)$ to $s\leq M_1$ satisfying $\underline{\Phi}(s)\geq0$ in $(-\infty, M_1)$, we have
$$d\underline{\Psi}''-(c_{\ast}-\epsilon)\underline{\Psi}'
+r(1-\underline{\Psi})(h\underline{\Phi}-\underline{\Psi})\geq0 $$
for all $s\in\mathbb{R}$. Moreover, $\underline{\Psi}'(M_1-)=0\leq \underline{\Psi}'(M_1+)$.

In view of \eqref{qq1} and $\psi_1(M_1)=0$, we have $\epsilon_1p_1(M_1)=\Psi_0(M_1)<\frac{\epsilon\alpha_1}{4k}.$
For $s\in(M_1,M_0)$, by \eqref{qq1}, we can fix $\epsilon_1>0$ sufficiently small so that
\bess
\begin{aligned}
\Phi_0''-(c_{\ast}-\epsilon)\Phi_0'+\Phi_0(1-k-\Phi_0+k\underline{\Psi})
=\epsilon\Phi_0'-k\epsilon_1p_1\Phi_0\geq(\frac{\alpha_1}{2}-k\epsilon_1p_1(M_1))\Phi_0
\geq0.
\end{aligned}
\eess
For $s\leq M_1$, we choose $\epsilon>0$ sufficiently small so that the quadratic equation
$\alpha^2-(c_*-\epsilon)\alpha+(1-k)=0$ has a root $\alpha_\epsilon \in (\alpha_1/2, \alpha_1)$.

Define
$$p_2(s)=0\mbox{ for }s\geq M_1,\quad p_2(s)=e^{\alpha_\epsilon s}\mbox{ for }s\leq M_1-1$$
and for $s\in[M_1-1,M_1]$, we define $p_2(s)$ so that $p_2(s)>0$, and $p_2(s)$ is $C^2$ everywhere.
We define
$$\underline{\Phi}(s)=\Phi_0(s)-\epsilon_2p_2(s),$$
with $\epsilon_2>0$ to be determined.

Since $\alpha_\epsilon<\alpha_1$, by \eqref{asymp1} we can find $M_1^\epsilon<M_1-1$ such that
$$\Phi_0'(s)>\alpha_\epsilon\Phi_0(s)\mbox{ for }s\leq M_1^\epsilon.$$
It follows that, for $s\leq M_1^\epsilon$,
\bess
\underline{\Phi}'(s)=\Phi_0'(s)-\epsilon_2p_2'(s)>\alpha_\epsilon\Phi_0(s)
-\epsilon_2\alpha_\epsilon p_2(s)=\alpha_\epsilon\underline{\Phi}(s).
\eess
Recall  that $\alpha_\epsilon>\alpha_1/2$. We now choose $\epsilon_2$ sufficiently small such that, for $x\in[M_1^\epsilon,M_1]$,
$$\underline{\Phi}(s)>0,\quad\underline{\Phi}'(s)>0$$
and
\bess
\begin{aligned}
&\underline{\Phi}''-(c_{\ast}-\epsilon)\underline{\Phi}'+\underline{\Phi}(1-k
-\underline{\Phi}+k\underline{\Psi})\\
=&-k\Phi_0\Psi_0+\epsilon\Phi_0'+\epsilon_2[c_{\ast}p_2'-p_2''-\epsilon p_2'
-p_2(1-k-2\Phi_0+\epsilon_2p_2)]\\
>&-k\Phi_0\Psi_0+\epsilon\alpha_\epsilon\Phi_0+\epsilon_2[c_{\ast}p_2'-p_2''-\epsilon p_2'
-p_2(1-k-2\Phi_0+\epsilon_2p_2)]\\
>&\frac{\epsilon}{4}\alpha_1\Phi_0+\epsilon_2[c_{\ast}p_2'-p_2''-\epsilon p_2'
-p_2(1-k-2\Phi_0+\epsilon_2p_2)]>0.
\end{aligned}
\eess
Due to $\alpha_\epsilon<\alpha_1$ and \eqref{asymp1}, we easily deduce $\lim\limits_{s\rightarrow-\infty}\frac{\Phi_0(s)}{e^{\alpha_\epsilon s}}=0.$
It follows that $$\underline{\Phi}(x)=e^{\alpha_\epsilon s}(\frac{\Phi_0(s)}{e^{\alpha_\epsilon s}}-\epsilon_2)<0$$ for all large negative $s.$
Since $\underline{\Phi}(M_1^\epsilon)>0$, by continuity, there exists $M_2<M_1^\epsilon$ such that
$$\underline{\Phi}(M_2)=0,\quad\underline{\Phi}(s)>0\mbox{ for }s\in(M_2,M_1^\epsilon].$$
Thus for $s\in(M_2,M_1^\epsilon)$, we have $-2\Phi_0+\epsilon_2p_2<-\Phi_0<0$ and
\bess
\begin{aligned}
&\underline{\Phi}''-(c_{\ast}-\epsilon)\underline{\Phi}'+\underline{\Phi}(1-k
-\underline{\Phi}+k\underline{\Psi})\\
>&\frac{\epsilon}{4}\alpha_1\Phi_0+\epsilon_2[(c_{\ast}-\epsilon)p_2'-p_2''
-p_2(1-k-2\Phi_0+\epsilon_2p_2)]\\
\geq&\frac{\epsilon}{4}\alpha_1\Phi_0+\epsilon_2[(c_{\ast}-\epsilon)p_2'-p_2''
-(1-k)p_2]=\frac{\epsilon}{4}\alpha_1\Phi_0>0.
\end{aligned}
\eess
Define
\bess
\quad\underline{\Phi}(s)=\left\{\begin{array}{ll}\medskip
\displaystyle
\Phi_0-\epsilon_2p_2(s),&s\geq M_2,\\
0,&s\leq M_2.
\end{array}\right.
\eess
It is easy to see that
$$\underline{\Phi}''-(c_{\ast}-\epsilon)\underline{\Phi}'+\underline{\Phi}(1-k
-\underline{\Phi}+k\underline{\Psi})\geq0$$
for all $s\in\mathbb{R}$. Moreover, $\underline{\Phi}'(M_2-)=0\leq \underline{\Phi}'(M_2+)$.

Finally, we accomplish the proof by the upper and lower solution argument.  Define
\[
(\underline{\phi}(s),\underline{\psi}(s)):=(\underline{\Phi}(s+M_2),\underline{\Psi}(s+M_2)).
\]
Then $(\underline{\phi}(s),\underline{\psi}(s))$ is a lower solution for \eqref{sect1} with $c=c_{\ast}-\epsilon$ associated with $\mathcal R:=[0,u^*]\times [0, hu^*]$. Moreover, it is easy to see that $(\overline{\phi}(s),\overline{\psi}(s))$ is a upper solution for \eqref{sect1} with $c=c_{\ast}-\epsilon$
associated with $\mathcal R$, where $(\overline{\phi}(s),\overline{\psi}(s))$ is given by \eqref{upfl}.

We next check that
\[
\mbox{$\sup_{t\leq s}\big(\underline\phi(t), \underline\psi(t)\big)\leq \big(\overline\phi(s),\overline\psi(s)\big)$ holds for all $s\in\mathbb R$.}
\]
By the monotonicity of $\underline \phi$ and $\underline \psi$, it suffices to show
\bes\lbl{lo<up}
\mbox{$\big(\underline\phi(s), \underline\psi(s)\big)\leq \big(\overline\phi(s),\overline\psi(s)\big)$  for all $s\in\mathbb R$.}
\ees

For $s>0$,
\[
\underline \psi(s)\leq \Psi_0(s+M_2)<hu^*=\overline \psi(s).
\]
Since $M_1-M_2>0$ and $\underline \psi(s)=0$ for $s<M_1-M_2$, we thus see that
\[
\underline\psi(s)\leq \overline\psi(s) \mbox{ for all } s\in\mathbb R.
\]

Clearly $\underline \phi(s)=0=\overline\phi(s)$ for $s\leq 0$. For $s>0$, due to $\underline\psi(s)\leq \underline\psi(\infty)=\Psi_0(\infty)=hu^*<1$,
we have
\[
0\leq \underline \phi''-(c_*-\epsilon)\underline \phi'+\underline \phi(1-k-\underline \phi+k\underline \psi)
\leq \underline \phi''+\underline \phi(1-\underline \phi).
\]
Moreover, $\underline\phi(\infty)=\Phi_0(\infty)=u^*=\overline \phi(\infty)$. Hence we can apply Lemma \ref{S21} to conclude that
$\underline \phi(s)\leq \overline \phi(s)$ for $s>0$.

We have thus proved \eqref{lo<up}. Clearly $\underline \varphi(s)\not\equiv 0$ and the nonlinearity functions in \eqref{sect1} satisfy
${\bf (A_1), \; (A_2), \; (A_3)}$.
We may now apply Proposition \ref{s7} to conclude  that \eqref{sect1} (with $\tilde\phi'>0$ and $\tilde\psi'>0$ relaxed to
$\tilde\phi'\geq 0$ and $\tilde\psi'\geq 0$) has a solution $(\phi,\psi)$ with $c=c_{\ast}-\epsilon$, which would imply $c_0^{\ast}\geq c_{\ast}-\epsilon$ if we can further prove $\phi'> 0$ and $\psi'> 0$.  But these strict inequalities follow easily from the strong maximum principle applied to the coorperate system satisfied by $(\phi', \psi')$. By the arbitrariness of $\epsilon$, it follows $c_0^\ast\geq c_{\ast}$.

{\bf Case 2}: $c_{\ast}=2\sqrt{1-k}$.

It follows from Lemma \ref{S22} that the following problem
\bess\left\{\begin{array}{ll}\medskip
\displaystyle
\tilde{\chi}''-(c_{\ast}-\epsilon)\tilde{\chi}'+\tilde{\chi}(1-k-\tilde{\chi})=0,\quad0<s<\infty,\\
\tilde{\chi}(0)=0
\end{array}\right.\lbl{}
\eess
has a unique strictly increasing solution $\tilde{\chi}$ satisfying $\tilde \chi(\infty)=1-k$.
Define
\bess
\underline{\Phi}(s):=\left\{\begin{array}{ll}\medskip
\displaystyle
0,&-\infty<s<0,\\
\delta\tilde{\chi}(s),&0\leq s<\infty,
\end{array}\right.
\quad\underline{\Psi}(s):=0,
\eess
with $\delta>0$ small such that $\underline{\Phi}(s)\leq\overline{\phi}(s)$ for $s\in\mathbb{R}$, where $\overline{\phi}(s)$ is given by \eqref{upfl}.
It is easy to verify that $(\underline{\Phi}(s),\underline{\Psi}(s))$ is a lower solution of \eqref{sect1} associated with $\mathcal R$.
Moreover, it is easy to see that $(\overline{\phi}(s),\overline{\psi}(s))$ is an upper solution for \eqref{sect1} with $c=c_{\ast}-\epsilon$
associated with $\mathcal R$, and $(\underline{\Phi}(s),\underline{\Psi}(s))\leq (\overline{\phi}(s),\overline{\psi}(s))$ for $s\in\mathbb{R}$. It follows from Proposition \ref{s7} and the strong maximum principle (applied to $(\phi',\psi')$ as in Case (i) above) that \eqref{sect1} has a solution $(\phi,\psi)\in\Sigma$ with $c=c_{\ast}-\epsilon$,
which implies $c_0^{\ast}\geq c_{\ast}-\epsilon$.  By the arbitrariness of $\epsilon$, it follows $c_0^\ast\geq c_{\ast}$.
\end{proof}

\begin{lem}\lbl{S28}For $c\geq c_{\ast}$, problem \eqref{sect1} has no solution. 
\end{lem}
\begin{proof}\, Suppose on the contrary that for some $c\geq c_{\ast}$, \eqref{sect1} has a solution $(\tilde{\phi},\tilde{\psi})$. By Proposition \ref{F6}, the system \eqref{sect1entire}  has a solution $(\tilde{\Phi},\tilde{\Psi})$ for such $c$. We are going to  derive a contradiction by making use of $(\tilde{\Phi},\tilde{\Psi})$.

We note that by repeating the arguments in the proof of Lemma \ref{S26}, the monotone increasing functions $\tilde{\phi},\tilde{\psi},\tilde{\Phi}$ and $\tilde{\Psi}$ can be expanded near $\infty$ in the form \eqref{expansion1}. In view of $\tilde{\Phi}'(s)>0$ and $\tilde{\Psi}'(s)>0$, there exists some $\eta_0>0$ such that
$$(\tilde{\Phi}(s+\eta),\tilde{\Psi}(s+\eta))\geq(\tilde{\phi}(s),\tilde{\psi}(s)),\quad\forall s\geq0,\quad\eta\geq\eta_0.$$
Clearly
\bes\lbl{ew0}\tilde{\Phi}(s+\eta)>0=\tilde{\phi}(s)\mbox{ for }s<0.\ees
Now we prove that
$$\tilde{\Psi}(s+\eta)\geq\tilde{\psi}(s)\mbox{ for }s\in\mathbb{R}\mbox{ and }\eta\geq\eta_0.$$
We only need to show this for $s<0$. Denote,  for $\eta\geq\eta_0$,
$$\tilde{\Phi}_\eta(s):=\tilde{\Phi}(s+\eta),\quad\tilde{\Psi}_\eta(s):=\tilde{\Psi}(s+\eta),
$$
and let
$$\hat{\Phi}_\eta(s)=1-\tilde{\Phi}_\eta(-s),\quad\hat{\Psi}_\eta(s)=1-\tilde{\Psi}_\eta(-s),\quad\psi(s)=1-\tilde{\psi}(-s).$$
Then
\bes\left\{\begin{aligned}\lbl{wave11}
&-c\hat{\Psi}_\eta'-d\hat{\Psi}_\eta''=r\hat{\Psi}_\eta(1-\hat{\Psi}_\eta-h\tilde{\Phi}_\eta)\leq r\hat{\Psi}_\eta(1-\hat{\Psi}_\eta-h\tilde{\phi}),\quad s\in\mathbb{R},\\
&-c\psi'-d\psi''=r\psi(1-\psi-h\tilde{\phi}),\quad s\in\mathbb{R},\\
&\hat{\Psi}_\eta(\infty)=1=\psi(\infty),\quad\hat{\Psi}_\eta(0)\leq\psi(0).
\end{aligned}
\right.\ees
Using Lemma \ref{S21} we deduce $\hat{\Psi}_\eta(s)\leq\psi(s)$ in $[0, \infty)$, and hence
$\tilde{\Psi}_\eta(s)\geq\tilde{\psi}(s)$ in $(-\infty,0]$.

We are now able to define
\bess\eta^\ast=\inf\left\{\eta_0\in\mathbb{R}:\quad\tilde{\Phi}_\eta(s)\geq\tilde{\phi}(s)\mbox{ in }[0,\infty), \quad\tilde{\Psi}_\eta(s)\geq\tilde{\psi}(s)\mbox{ in }\mathbb{R},\quad\forall \eta\geq\eta_0\right\}.\eess
We claim that $\eta^\ast=-\infty$. Otherwise, $\eta^\ast$ is a finite real number, and by continuity,
\bess\tilde{\Phi}_{\eta^\ast}(s)\geq\tilde{\phi}(s)\mbox{ in }[0,\infty), \quad\tilde{\Psi}_{\eta^\ast}(s)\geq\tilde{\psi}(s)\mbox{ in }\mathbb{R}.\eess
The first inequality of \eqref{wave11} still holds for $\eta=\eta^\ast$, and this inequality is strict for $s<0$ due to \eqref{ew0}. Hence $\tilde{\Psi}_{\eta^\ast}(s)\not\equiv\tilde{\psi}(s)$, and by the strong maximum principle we obtain
$$\tilde{\Psi}_{\eta^\ast}(s)>\tilde{\psi}(s)\mbox{ for }s\in\mathbb{R}.$$
We now have
\bess\begin{aligned}\lbl{wave1}
&c\tilde{\Phi}_{\eta^\ast}'-\tilde{\Phi}_{\eta^\ast}''=\tilde{\Phi}_{\eta^\ast}(1-k-\tilde{\Phi}_{\eta^\ast}
+k\tilde{\Psi}_{\eta^\ast})\geq \tilde{\Phi}_{\eta^\ast}(1-k-\tilde{\Phi}_{\eta^\ast}+k\tilde{\psi}),\quad s\in\mathbb{R}^+,\\
&c\tilde{\phi}'-\tilde{\phi}''=\tilde{\phi}(1-k-\tilde{\phi}+k\tilde{\psi}),\quad s\in\mathbb{R}^+,\\
&\tilde{\Phi}_{\eta^\ast}(0)\geq\psi(0),\quad\tilde{\Phi}_{\eta^\ast}(\infty)=\tilde{\phi}(\infty)=u^\ast.
\end{aligned}\eess
Using Lemma \ref{S21} we deduce
$$\tilde{\Phi}_{\eta^\ast}(s)\geq\tilde{\phi}(s)\mbox{ for }s\in[0,\infty).$$
We may now use the expansion of $(\tilde{\Phi}_{\eta^\ast}-\tilde{\phi},\tilde{\Psi}_{\eta^\ast}-\tilde{\psi})$ near $s=\infty$ as the proof of Lemma \ref{S27} to derive that
\bess
\tilde{\Phi}_{\eta^\ast-\epsilon}(s)\geq\tilde{\phi}(s),\quad\tilde{\Psi}_{\eta^\ast-\epsilon}(s)>\tilde{\psi}(s) \mbox{ for }s\in[0,\infty)\mbox{ and some small }\epsilon>0.
\eess
It then follows from the monotonicity of $\tilde \Phi$ and $\tilde \Psi$ that  for all $\eta\geq\eta^\ast-\epsilon$,
$$\tilde{\Phi}_\eta(s)\geq\tilde{\phi}(s)\mbox{ in }[0,\infty), \quad\tilde{\Psi}_\eta(s)\geq\tilde{\psi}(s)\mbox{ in }\mathbb{R},$$
which contradicts the definition of $\eta^\ast$. Hence, $\eta^\ast=-\infty$. The fact $\eta^\ast=-\infty$ implies $\tilde{\Phi}(s+\eta)\geq\tilde{\phi}(s)\mbox{ in }[0,\infty)$ for all $\eta\in\mathbb{R}$. For any fixed $s>0$, letting $\eta\rightarrow-\infty$ and using $\tilde{\Phi}(-\infty)=0$  we obtain $0\geq \tilde{\phi}(s)$. This is a contradiction to the fact that $(\tilde{\phi},\tilde{\psi})$ is a solution of \eqref{sect1}.
\end{proof}

\begin{lem}\lbl{c_*}$c_0^*=c_*$.
\end{lem}
\begin{proof}
Lemmas \ref{aa1} and \ref{S25} imply that \eqref{sect1} has a solution for every $c\in [0, c_0^*)$. Therefore Lemma \ref{S28}
implies $c_*\geq c_0^*$. In view of Lemma \ref{aa1}, we must have $c_0^*=c_*$.
\end{proof}

\begin{lem}\lbl{aa0}$c_0^\ast\leq2\sqrt{u^\ast}$.
\end{lem}
\begin{proof}\, Suppose on the contrary that $c_0^*>2\sqrt{u^*}$. Then  system \eqref{sect1} has a solution for some $c>2\sqrt{u^\ast}$. The monotonicity of $(\tilde{\phi}(s),\tilde{\psi}(s))$ implies that $(\tilde{\phi}(s),\tilde{\psi}(s))\leq(u^\ast,hu^\ast)$ on $\mathbb{R}^+$.

We claim that
$\tilde{\phi}'(s)$ and $\tilde{\phi}''(s)$ are uniformly bounded on $\mathbb{R}^+$. Let $\beta=\max\{1+k,r(1+h)\}$; it follows from Lemma 6.1 and \eqref{s002} that
\bess\begin{aligned}
|\tilde{\phi}'(s)|&=\left|\frac{1}{\lambda_{2}-\lambda_{1}}\left[\int_{0}^sK_{1s}(\xi,s)\tilde{\phi}(\xi)(\beta+1
-k-\tilde{\phi}(\xi)+k\tilde{\psi}(\xi))d\xi\right.\right.\\
&\quad\;\;\;\;\;\;\;\;\;\;+\left.\left.\int_s^{\infty}K_{2s}(\xi,s)\tilde{\phi}(\xi)(\beta+1-k
-\tilde{\phi}(\xi)+k\tilde{\psi}(\xi))d\xi\right]\right|\\
&\leq\frac{2(\beta+1+k+u^\ast+khu^\ast)u^\ast}{\lambda_{2}-\lambda_{1}}:=\hat{C}_1.
\end{aligned}\eess
Using the boundedness of $\tilde{\phi},\tilde{\psi},\tilde{\phi}'$  and \eqref{sect1}, we obtain that
\bess\begin{aligned}
|\tilde{\phi}''|&=|c\tilde{\phi}'-\tilde{\phi}(1-k-\tilde{\phi}+k\tilde{\psi})|\\
&\leq c|\tilde{\phi}'|+|\tilde{\phi}(1-k-\tilde{\phi}+k\tilde{\psi})|\\
&\leq c\hat{C}_1+u^\ast(1+k+u^\ast+khu^\ast):=\bar{C}_1.
\end{aligned}\eess
Thus $|\tilde{\phi}'|,|\tilde{\phi}''|<C$ with $C:=\max\{\hat{C}_1,\bar{C}_1\}$.

Thanks to the uniform boundedness of $\tilde{\phi},\tilde{\psi},\tilde{\phi}'$ and $\tilde{\phi}''$, the integrals $\int_0^\infty\tilde{\phi}(s)\tilde{\psi}(s)e^{-\mu s}ds$ and
$\int_0^\infty\tilde{\phi}^{(l)}e^{-\mu s}ds(l=0,1,2)$ are well defined for any $\mu>0$. In view of $c>2\sqrt{u^\ast}$, we know that
$$\mu^2-c\mu+u^\ast=0$$
has two positive roots, say $\tilde{\mu}_1,\tilde{\mu}_2$ with $0<\tilde{\mu}_1<\tilde{\mu}_2$.
Now, choosing $\mu\in(\tilde \mu_1,\tilde\mu_2)$, multiplying the first equation in
 \eqref{sect1} by $e^{-\mu s}$ and integrating from $0$ to $\infty$, we obtain
\bes\begin{aligned}
&\tilde{\phi}'(0)+\int_0^\infty\tilde{\phi}^2e^{-\mu s}ds\\
=&(\mu^2-c\mu)\int_0^\infty\tilde{\phi}e^{-\mu s}ds+\int_0^\infty(1-k+k\tilde{\psi})\tilde{\phi} e^{-\mu s}ds\\
\leq&(\mu^2-c\mu+u^\ast)\int_0^\infty\tilde{\phi}e^{-\mu s}ds<0.
\end{aligned}
\lbl{nonexistence}
\ees
Since $\tilde{\phi}'(0)>0$ by the Hopf boundary Lemma, we have $\tilde{\phi}'(0)+\int_0^\infty\tilde{\phi}^2e^{-\mu s}ds>0$, which contradicts \eqref{nonexistence}.
\end{proof}

\begin{lem}\lbl{T9}Let $(\tilde{\phi}_c,\tilde{\psi}_c)$ denote the unique solution of \eqref{sect1}. Then $0\leq c_1<c_2< c_0^\ast $ implies
$$\tilde{\phi}_{c_1}'(0)>\tilde{\phi}_{c_2}'(0),\quad \tilde{\phi}_{c_1}(s)>\tilde{\phi}_{c_2}(s)\mbox{ in }\mathbb{R}^+,
\quad \tilde{\psi}_{c_1}(s)>\tilde{\psi}_{c_2}(s)\mbox{ in }\mathbb{R}.$$
\end{lem}
\begin{proof} From the proof of Lemma \ref{S25} and the uniqueness of solutions to \eqref{sect1}, we have $(\tilde{\phi}_{c_2}(s),\tilde{\psi}_{c_2}(s))\leq(\overline{\phi}(s),\overline{\psi}(s))$, where $(\overline{\phi}(s),\overline{\psi}(s))$ is given by
\eqref{upfl}.
Moreover, due to $0\leq c_1<c_2$ and $\tilde\phi'_{c_2}>0$ in $\mathbb R^+$ and $\tilde\psi'_{c_2}>0$ in $\mathbb R$, we have
\bess\left\{
\begin{aligned}
&-\tilde{\phi}_{c_2}''+c_1\tilde{\phi}_{c_2}'<\tilde{\phi}_{c_2}(1-k-\tilde{\phi}_{c_2}
+k\tilde{\psi}_{c_2}),&0<s<\infty,\\
&-d\tilde{\psi}_{c_2}''+c_1\tilde{\psi}_{c_2}'< r(1-\tilde{\psi}_{c_2})(h\tilde{\phi}_{c_2}-\tilde{\psi}_{c_2}),&-\infty<s<\infty.
\end{aligned}
\right.\eess
Hence, it follows from Proposition \ref{s7} and Lemma \ref{S27} that $\tilde{\phi}_{c_1}(s)\geq\tilde{\phi}_{c_2}(s)$ in $\mathbb{R}^+$ and
$\tilde{\psi}_{c_1}(s)\geq\tilde{\psi}_{c_2}(s)$ in $\mathbb{R}$. Furthermore, the strong maximum principle yields
$\tilde{\phi}_{c_1}(s)>\tilde{\phi}_{c_2}(s)$ for $s>0$,
$\tilde{\psi}_{c_1}(s)>\tilde{\psi}_{c_2}(s)$ for $s\in\mathbb{R}$.
Let $\widetilde{\phi}=\tilde{\phi}_{c_1}-\tilde{\phi}_{c_2}$. Then
$$-\widetilde{\phi}''+c_2\widetilde{\phi}'>\widetilde{\phi}(1-k-\tilde{\phi}_{c_1}-\tilde{\phi}_{c_2}
+k\tilde{\phi}_{c_1}),\quad\widetilde{\phi}(s)>0\mbox{ for }s>0,\quad\widetilde{\phi}(0)=0.$$
By the Hopf boundary lemma, we deduce $\widetilde{\phi}'(0)>0$, that is, $\tilde{\phi}_{c_1}'(0)>\tilde{\phi}_{c_2}'(0)$.
\end{proof}

\begin{lem}\lbl{T10}Let $(\tilde{\phi}_c,\tilde{\psi}_c)$ be the unique monotone solution of \eqref{sect1}. Then the mapping $c\longmapsto(\tilde{\phi}_c,\tilde{\psi}_c)$ is continuous from $[0,c_0^\ast)$ to $C_{loc}^2([0,\infty))\times C_{loc}^2(\mathbb{R})$. Moreover,
\bess\lim_{c\rightarrow c_0^\ast-}(\tilde{\phi}_c,\tilde{\psi}_c)=(0,0)\mbox{ in }C_{loc}^2([0,\infty))\times C_{loc}^2(\mathbb{R}).\lbl{critical}\eess
\end{lem}
\begin{proof}
Suppose $\{c_i\}$ is a sequence in $[0,c_0^\ast)$ such that $c_i\rightarrow \hat{c}\in[0,c_0^\ast]$ as $i\rightarrow\infty$.
Let $(\tilde{\phi}_{c_i},\tilde{\psi}_{c_i})$ be the solution of \eqref{sect1} with $c=c_i$. We claim that $(\tilde{\phi}_{c_i},\tilde{\psi}_{c_i})$ has a subsequence that converges to $(\tilde{\phi}_{\hat{c}},\tilde{\psi}_{\hat{c}})$ in $C_{loc}^2([0,\infty))\times C_{loc}^2(\mathbb{R})$, which clearly  implies  the continuity of the mapping $c\longmapsto(\tilde{\phi}_c,\tilde{\psi}_c)$.

Firstly, we consider the case $\hat{c}<c_0^\ast$. Let $\bar{c}\in(\hat{c},c_0^\ast)$.  Then $c_i\in[0,\bar{c})$ for all large $i$,
and without loss of generality we assume that this is the case for all $i\geq 1$. For simplicity, we denote $(\tilde{\phi}_{c_i},\tilde{\psi}_{c_i})$ by $(\tilde{\phi}_{i},\tilde{\psi}_{i})$. Rewrite equation \eqref{sect1} in the integral form of \eqref{l01} and \eqref{l02}.
Noting that $\tilde{\phi}_{i}$ and $\tilde{\psi}_{i}$ are uniformly bounded, similar arguments as in Lemma \ref{aa0} indicate that $|\tilde{\phi}_{i}'|$ and $|\tilde{\phi}_{i}''|$ are bounded for all $i$ and $s\in\mathbb{R}^+$. Moreover, by similar arguments as in the proof of Lemma \ref{aa0} again, we can prove  $|\tilde{\psi}_{i}'|$ and $|\tilde{\psi}_{i}''|$ are bounded for all $i$ and $s\in\mathbb{R}$. Differentiating both sides of \eqref{sect1} with respect to $s$, applying the uniform boundedness of $\tilde{\phi}_{i}^{(j)}$ and $\tilde{\psi}_{i}^{(j)}(j=0,1,2)$, we have $|\tilde{\phi}_{i}'''|$ and $|\tilde{\psi}_{i}'''|$ are bounded for $s\in\mathbb{R}^+$ and $s\in\mathbb{R}$, respectively. Hence, $\{\tilde{\phi}_{i}^{(j)}\}$ and $\{\tilde{\psi}_{i}^{(j)}\}(j=0,1,2)$ are uniformly bounded and equi-continuous  for $s\in\mathbb{R}^+$ and $s\in\mathbb{R}$, respectively. Using Arzela-Ascoli's theorem, the nested subsequence argument and Lebesgue's dominated convergence theorem, there is a subsequence $c_{i_k}$ of $\{c_i\}$ such that
$(c_{i_k},\tilde{\phi}_{i_k}^{(j)},\tilde{\psi}_{i_k}^{(j)})\rightarrow(\hat{c},\hat{\phi}^{(j)},\hat{\psi}^{(j)})$
uniformly as $k\rightarrow\infty$ in $C_{loc}^2([0,\infty))\times C_{loc}^2(\mathbb{R})$. Moreover, $(\hat{\phi},\hat{\psi})$ solves \eqref{sect1} with $c=\hat{c}$, except that we only have $\hat{\phi}'\geq0$ and $\hat{\psi}'\geq0$.
Using Lemma \ref{T9}, the required asymptotic behavior of $(\hat{\phi},\hat{\psi})$ at $\pm\infty$ follows from $$(\tilde{\phi}_{\bar{c}},\tilde{\psi}_{\bar{c}})\leq(\hat{\phi},\hat{\psi})\leq(\overline{\phi},\overline{\psi}).$$
Applying the strong maximum principle to the system satisfied by $(\hat{\phi}',\hat{\psi}')$, we deduce $\hat{\phi}'>0$ in $[0,\infty)$ and $\hat{\psi}'>0$ in $\mathbb{R}$. Thus $(\hat{\phi},\hat{\psi})$ is a solution of  \eqref{sect1} with $c=\hat{c}$.
By uniqueness, we have $(\hat{\phi},\hat{\psi})=(\tilde{\phi}_{\hat{c}},\tilde{\psi}_{\hat{c}})$.

It remains to consider the case $\hat{c}=c_0^\ast$. Repeating the above arguments, we conclude that, passing to a subsequence,
$$(c_{i_k},\tilde{\phi}_{i_k},\tilde{\psi}_{i_k})\rightarrow( c_0^\ast ,\hat{\phi}_\ast,\hat{\psi}_\ast)\mbox{ in }C_{loc}^2([0,\infty))\times C_{loc}^2(\mathbb{R})\mbox{ as }k\rightarrow\infty$$
and $(\hat{\phi}_\ast,\hat{\psi}_\ast)$ solves \eqref{sect1} with $c= c_0^\ast $, except that we only have $\hat{\phi}_\ast'\geq0$ in $[0,\infty)$ and $\hat{\psi}_\ast'\geq0$ in $\mathbb{R}$. If $\hat{\phi}_\ast\equiv0$, then $\hat{\psi}_\ast$ satisfies
$$d\hat{\psi}_\ast''- c_0^\ast \hat{\psi}_\ast'=r\hat{\psi}_\ast(1-\hat{\psi}_\ast).$$
Let $\tilde{\psi}_\ast=1-\hat{\psi}_\ast$. Then $\tilde{\psi}_\ast$ satisfies
$$-d\tilde{\psi}_\ast''\geq -d\hat{\psi}_\ast''+ c_0^\ast \hat{\psi}_\ast'= r\tilde{\psi}_\ast(1-\tilde{\psi}_\ast).$$
For large $L>1$, assume $u_L$ is the unique positive solution of
$$-du''=ru(1-u),\quad u(\pm L)=0.$$
It is well known that $u_L\rightarrow1$ in $C_{loc}^2(\mathbb{R})$ as $L\rightarrow\infty$. By Lemma 2.1 of \cite{DM}, we have $u_L\leq\tilde{\psi}_\ast\leq1$ in $[-L,L]$. Letting $L\rightarrow\infty$ we obtain $\tilde{\psi}_\ast=1$, as we wanted.
Next, assume that $\hat{\phi}_\ast\not\equiv0$. Let $(\hat{\phi}_\ast,\hat{\psi}_\ast)$ be a solution of \eqref{sect1} with $c=c_0^\ast$. Then we may repeat the proof of Lemma \ref{S28} to conclude $\hat{\phi}_\ast\leq0$, a contradiction.
\end{proof}

\begin{lem}\lbl{tlem13}Let $(\tilde{\phi}_c,\tilde{\psi}_c)$ be the unique monotone solution of \eqref{sect1}. For any $\gamma>0$, there exists a unique $c=c(\gamma)\in(0, c_0^\ast )$ such that $\gamma\tilde{\phi}_c'(0)=c$. Moreover, $\gamma\longmapsto c(\gamma)$ is strictly increasing and $\lim_{\gamma\rightarrow\infty}c(\gamma)= c_0^\ast $.
\end{lem}

\begin{proof}\,By Lemma \ref{T10} and Lemma \ref{T9}, for fixed $\gamma>0$, the function $p(c,\gamma)=\gamma\tilde{\phi}_c'(0)-c$ is continuous and strictly decreasing for $c\in[0, c_0^\ast )$. Note that $\lim_{c\rightarrow c_0^\ast }p(c,\gamma)=- c_0^\ast <0$ and $p(0,\gamma)=\gamma\tilde{\phi}_{0}'(0)>0$. Therefore, there exists a unique $c=c(\gamma)\in(0, c_0^\ast )$ such that $p(c,\gamma)=0$, i.e. $\gamma\phi_c'(0)=c$. Moreover, note that $p(c,\gamma)$ is strictly increasing in $\gamma$ for any given $c\in(0, c_0^\ast )$. Hence, $c(\gamma)$ is strictly increasing in $\gamma$. For any $\epsilon>0$ and $c\in[0, c_0^\ast -\epsilon]$, we have $p(c,\gamma)\geq p( c_0^\ast -\epsilon,\gamma)\rightarrow\infty$ as $\gamma\rightarrow\infty$. It follows that
$ c_0^\ast -\epsilon<c(\gamma)< c_0^\ast $ for all large $\gamma$,  which means that $\lim_{\gamma\rightarrow\infty}c= c_0^\ast $.
\end{proof}

Theorem \ref{F4} now follows directly from Lemmas \ref{S25}--\ref{tlem13}.
{\setlength{\baselineskip}{16pt}{\setlength\arraycolsep{2pt} \indent
\section{The spreading-vanishing dichotomy}

\setcounter{equation}{0}

We prove Theorems 1.1 and 1.2 in this section.
Let us  recall that for the problem
\bess\left\{
\begin{aligned}
&u_t=du_{xx}+u(a-bu),&0<x<h(t),&\quad t>0,\\
&u_x(0,t)=0,u(x,t)\equiv0,&h(t)\leq x,&\quad t>0,\\
&h'(t)=-\nu u_x(h(t),t),&\nu>0,&\quad t>0,\\
&h(0)=h_0>0,\quad u(x,0)=u_0(x),&0\leq x\leq h_0,&
\end{aligned}
\right.\lbl{(se1)}
\eess
the following result holds.
\begin{lem}\label{t3} {\rm (\cite{DL10})} If $h_0\geq \frac{\pi}{2}\sqrt{\frac{d}{a}}$, then spreading always happens. If $h_0<\frac{\pi}{2}\sqrt{\frac{d}{a}}$, then there exists $\nu^\ast>0$ depending on $u_0$ such that vanishing happens when $\nu\leq\nu^\ast$, and spreading happens when $\nu>\nu^\ast$.
\end{lem}

\begin{lem}\lbl{t4}Let $(u,v,g)$  be the solution of \eqref{section1-1}. If $\lim_{t\rightarrow\infty}g(t)=g_\infty<\infty$, then the solution of equation \eqref{section1-1} satisfies
\bess
\lim_{t\rightarrow\infty}\|u(\cdot, t)\|_{C([0,g(t)])}=0,\;\lim_{t\to\infty}v(\cdot,t)=1 \mbox{ in } C_{loc}([0,\infty)).
\eess
\end{lem}
\begin{proof} The proof is similar  to  the proof of Lemma 4.6 in \cite{DL14}. For readers' convenience, we give the details here. Define
$$s:=\frac{g_0x}{g(t)},\quad \hat{u}(s,t):=u(x,t),\quad \hat{v}(s,t):=v(x,t).$$
By direct calculation,
$$u_t=\hat{u}_t-\frac{g'(t)}{g(t)}s\hat{u}_s,\quad u_x=\frac{g_0}{g(t)}\hat{u}_s,
\quad u_{xx}=\frac{g_0^2}{g^2(t)}\hat{u}_{ss}.$$
Hence $\hat u$ satisfies
$$\left\{
\begin{aligned}
&\hat{u}_{t}-\frac{g_0^2}{g^2(t)}\hat{u}_{ss}-\frac{g'(t)}{g(t)}s\hat{u}_s
=\hat{u}(1-\hat{u}-k\hat{v}),&0< s<g_0,&\quad t>0,\\
&\hat{u}_{s}(0,t)=\hat{u}(g_0, t)=0,&&\quad t>0,\\
&\hat{u}(s, 0)=u_{0}(s),&0\leq s\leq g_0.&
\end{aligned}
\right.\eqno{}$$
 By Proposition 2.1, there exists $M>0$ such that
$$\|1-\hat{u}-k\hat{v}\|_{L^\infty}\leq 1+(1+k)M,\quad \left\|\frac{g'(t)}{g(t)}\right\|_{L^\infty}\leq \frac{M}{g_0}.$$
Since $g_0\leq g(t)<g_\infty<\infty$,  the differential operator is uniformly parabolic. Therefore we can apply standard $L^p$ theory to obtain, for any $p>1$,
$$\|\hat{u}\|_{W_p^{2,1}([0,g_0]\times[0,2])}\leq C_1,$$
where $C_1$ is a constant depending on $p,g_0,M$ and $\|u_{0}\|_{C^{1+\alpha}[0,g_0]}$.
For each $T\geq1$, we can apply the partial interior-boundary estimate over $[0,g_0]\times[T, T+2]$ to
obtain $\|\hat{u}\|_{W_p^{2,1}([0,g_0]\times[T+1/2, T+2])}\leq C_2$ for some constant $C_2$ depending on $\alpha, g_0,M$ and $\|u_{0}\|_{C^{1+\alpha}[0,g_0]}$, but
independent of $T$. Therefore, we can use the Sobolev imbedding theorem to obtain, for any $\alpha\in(0,1),$
\bes\|\hat{u}\|_{C^{1+\alpha, (1+\alpha)/2}([0,g_0]\times[0,\infty))}\leq C_3,\lbl{eshu}\ees
where $C_3$ is a constant depending on $\alpha, g_0,M$ and $\|u_{0}\|_{C^{1+\alpha}[0,g_0]}$.
Similarly we may use interior estimates to the equation of $\hat{v}$ to obtain
\bes\|\hat{v}\|_{C^{1+\alpha,(1+\alpha)/2}([0,g_0]\times[0,+\infty))}\leq C_4,\lbl{eshv}\ees
where $C_4$ is a constant depending on $\alpha,g_0,M$ and $\|v_{0}\|_{C^{1+\alpha}[0,g_\infty+1]}$.

Since
\[
g'(t)=-\gamma u_x(g(t),t)=-\gamma \frac{g_0}{g(t)}\hat u_s(g_0,t),
\]
it follows that there exists a constant $\tilde{C}$ depending on
$\alpha,\gamma,g_0,\|(u_{0},v_0)\|_{C^{1+\alpha}[0,g_0]}$ and $g_\infty$ such that
\bes\|g\|_{C^{1+\alpha/2}([0,+\infty))}\leq \tilde{C}.\lbl{se3l1}\ees

For contradiction, we assume that
$$\limsup_{t\rightarrow+\infty}\|u(\cdot ,t)\|_{C([0,g(t)])}=\delta>0.$$
Then there exists a sequence $(x_k,t_k)$ with $0\leq x_k<g(t_k),\ 1<t_k<\infty$ such that $u(x_k,t_k)\geq\frac{\delta}{2}>0$ for all $k\in\mathbb{N}$,
and $t_k\rightarrow+\infty$ as $k\rightarrow+\infty$. By \eqref{se3l1}, we know $|u_x(g(t),t)|$ is uniformly bounded for $t\in[0,+\infty)$, and there exists $\sigma>0$ such that $x_k\leq g(t_k)-\sigma$ for all $k\geq1$. Therefore there exists a subsequence of $\{x_k\}$ that converges to
some $x_0\in[0,g_\infty-\sigma]$. Without loss of generality, we may assume $x_k\rightarrow x_0$ as $k\rightarrow+\infty$, which leads to
$s_k=\frac{g_0x_k}{g(t_k)}\rightarrow s_0=\frac{g_0x_0}{g_\infty}<g_0.$

Set
$$\hat{u}_k(s,t)=\hat{u}(s,t_k+t),\quad \hat{v}_k(s,t)=\hat{v}(s,t_k+t)$$
for $(s,t)\in [0,g_0]\times[-1, 1].$  It follows from \eqref{eshu} and \eqref{eshv} that
$\{(\hat{u}_k,\hat{v}_k)\}$ has a subsequence $\{(\hat{u}_{k_i},\hat{v}_{k_i})\}$ such that
$$\|(\hat{u}_{k_i},\hat{v}_{k_i})-(\hat{u}^*,\hat{v}^*)\|_{C^{1+\alpha',(1+\alpha')/2}([0,g_0]\times[-1, 1])}\rightarrow0 \mbox{ as }i\rightarrow+\infty,$$
where $\alpha'\in(0, \alpha)$.
Since $\|g\|_{C^{1+\alpha/2}([0,+\infty))}\leq \tilde{C}$, $g'(t)>0$ and $g(t)\leq g_\infty$, we necessarily have $g'(t)\rightarrow0$ as $t\rightarrow\infty.$
Hence, $(\hat{u}^*, \hat{v}^*)$ satisfies
$$\begin{cases}
\hat{u}^*_t-(\frac{g_0}{g_\infty})^2\hat{u}^*_{ss}=\hat{u}^*(1-\hat{u}^*-k\hat{v}^*),& 0\leq s<g_0,\quad t\in(-1, 1),\\
\hat{u}^*_s(0, t)=\hat{u}^*(g_0, t)=0,&\quad\quad \quad \quad \quad \quad  t\in[-1, 1].
\end{cases}$$
Clearly, $\hat{u}_k(s_k,0)=u(x_k, t_k)\geq\frac{\delta}{2}$, Hence, we have $\hat{u}^*(s_0, 0)\geq\frac{\delta}{2}.$
By the maximum principle,  $\hat{u}^*>0$ in $[0,g_0)\times(-1, 1)$. Thus we can apply the Hopf boundary lemma to conclude that $\theta_0:=\hat{u}^*_s(g_0, 0)<0$. It follows that $u_{x}(g(t_{k_i}),t_{k_i})=\partial_s\hat{u}_{k_i}(g_0,0)\frac{g_0}{g(t_{k_i})}\leq\frac{\theta_0g_0}{2g_\infty}<0$ for all large $i$, and hence
$$g'(t_{k_i})=-\gamma u_{x}(g(t_{k_i}), t_{k_i})\geq-\gamma\theta_0/2>0$$
for all large $i$. On the other hand, recalling that $g'(t)\rightarrow0$ as $t\rightarrow+\infty$, we obtain a contradiction.
Hence we must have
\[
\lim_{t\rightarrow+\infty}\|u(\cdot, t)\|_{C([0,g(t)])}=0.
\]
Using this fact and a simple comparison argument we easily deduce $\lim_{t\to\infty} v(\cdot, t)=1$ uniformly in any compact subset of $[0,\infty)$.
\end{proof}

\begin{lem}\lbl{t6}Let $(u,v,g)$  be the solution of \eqref{section1-1} and suppose $g_\infty=\infty$. Then
\bess
(u(x,t),v(x,t))\rightarrow\left(u^\ast,v^\ast\right) \mbox{ as }t\rightarrow\infty
\eess
uniformly for $x$ in any compact subset of $[0,\infty)$.
\end{lem}
\begin{proof} We define
\[
\overline{u}_1= \overline{v}_1=1,\; \underline u_1=1-h,\; \underline v_1=1-k.
\]
Then define inductively for $n\geq 1$,
\[
\overline{u}_{n+1}=1-k\underline{v}_{n},\ \overline{v}_{n+1}=1-h\underline{u}_{n},\ \underline{u}_{n+1}=1-k\overline{v}_{n},\ \underline{v}_{n+1}=1-h\overline{u}_{n}.
\]
It is easily checked that $\{\overline u_n\}$ and $\{\overline v_n\}$ are decreasing, $\{\underline u_n\}$ and $\{\underline v_n\}$ are increasing, and
\bes\lbl{u_n-v_n}
\lim_{n\to\infty} (\overline u_n, \overline v_n)=\lim_{n\to\infty} (\underline u_n, \underline v_n)=(u^*, v^*).
\ees
We claim that, for every $n\geq 1$,
\bes\begin{aligned}
&\liminf_{t\rightarrow\infty}u(x,t)\geq \underline{u}_{n},\; \liminf_{t\rightarrow\infty}v(x,t)\geq \underline{v}_{n},
\\
&\limsup_{t\rightarrow\infty}u(x,t)\leq\overline{u}_{n},\;\limsup_{t\rightarrow\infty}v(x,t)\leq\overline{v}_{n},
\end{aligned}
\lbl{it}\ees
 uniformly in any compact subset of $[0,\infty)$.
 The conclusion of the Lemma clearly follows directly from \eqref{it} and \eqref{u_n-v_n}. So it suffices to prove \eqref{it}. We do that by an induction argument.

{\bf Step 1}. \eqref{it} holds for $n=1$.

It follows from the comparison principle that $u(x,t)\leq\hat{u}_1(t)$ for $t>0$ and $x\in[0,g(t)]$, where $\hat{u}_1(t)$ satisfies
\bess
\left\{
\begin{aligned}
&\frac{d\hat{u}_{1}}{dt}=\hat{u}_1(1-\hat{u}_1),&t>0,\\
&\hat{u}_1(0)=\|u_0\|_\infty.
\end{aligned}
\right.
\eess
 Clearly, $\lim_{t\rightarrow\infty}\hat{u}_1(t)=1$. Hence,
\bes
\limsup_{t\rightarrow\infty}u(x,t)\leq1=\overline{u}_1\mbox{ uniformly for }x\in[0,\infty),\lbl{it1}
\ees
By the same argument as above, one gets
\bes
\limsup_{t\rightarrow\infty}v(x,t)\leq1=\overline{v}_1\mbox{ uniformly for }x\in[0,\infty).\lbl{it2}
\ees

For any given $l>\max\left\{g_0,\frac{\pi}{2}\frac{1}{\sqrt{1-k}},\frac{\pi}{2}\sqrt{\frac{d}{r(1-h)}}\right\}$.
In view of \eqref{it1}, \eqref{it2} and $g_\infty=\infty$, for any small $\epsilon>0$, there exists $t_1>0$ such that
$g(t)>l$ for $t\geq t_1$ and $u(x,t)<\overline{u}_1+\epsilon$, $v(x,t)<\overline{v}_1+\epsilon$
for $x\in [0, l]$, $t>t_1$.
It follows that
\bess\left\{
\begin{aligned}
&v_t\geq dv_{xx}+rv(1-v-h(1+\epsilon)),&0<x<l,\quad t>t_1,\\
&v_x(0,t)=0,\quad v(l,t)>0,&0\leq x<l,\quad t>t_1,\\
\end{aligned}
\right.\lbl{}\eess
which implies that $v$ is an upper solution to the problem
\bess\left\{
\begin{aligned}
&\hat{v}_t=d\hat{v}_{xx}+r\hat{v}(1-\hat{v}-h(1+\epsilon)),&&0<x<l, & t>t_1,\\
&\hat{v}_x(0,t)=0,\quad \hat{v}(l,t)=0,&&&  t>t_1,\\
&\hat v(x, t_1)=v(x, t_1), && 0\leq x\leq l.&
\end{aligned}
\right.\lbl{}\eess
Hence
\[
v(x,t)\geq \hat v(x,t) \mbox{ for $x\in[0,l]$ and $t>t_1$.}
\]
In view of $l>\frac{\pi}{2}\sqrt{\frac{d}{r(1-h)}}$, it is well known that $\lim_{t\rightarrow\infty}\hat{v}(x,t)=\hat{v}^\ast(x)$, where $\hat{v}^\ast(x)$ is the unique positive solution of
\bess\left\{
\begin{aligned}
&d\hat{v}^\ast_{xx}+r\hat{v}^\ast(1-\hat{v}^\ast-h(1+\epsilon))=0,&0<x<l,\\
&\hat{v}^\ast_x(0)=0,\quad \hat{v}^\ast(l)=0.
\end{aligned}
\right.\lbl{}\eess
On the other hand, $\hat{v}^\ast\rightarrow1-h(1+\epsilon)$ uniformly in any compact subset of $[0,\infty)$ as $l\rightarrow\infty$
(see, for example, Lemma 2.2 in \cite{DM}). Thanks to the arbitrariness of $l$ and $\epsilon$, we thus obtain from $v(x,t)\geq \hat v(x,t)$ in $[0,l]\times (t_1,\infty)$ that
\bess\liminf_{t\rightarrow\infty}v(x,t)\geq1-h=\underline{v}_1\mbox{ uniformly in any compact subset of }[0,\infty).\lbl{it3}\eess

Similarly, we have
\bess\left\{
\begin{aligned}
&u_t-u_{xx}\geq u(1-u-k(1+\epsilon)),&0<x<l,&\quad t>t_1,\\
&u_x(0,t)=0,\quad u(l,t)>0,&&\quad t>t_1,
\end{aligned}
\right.\eess
which leads to
\bess\liminf_{t\rightarrow\infty}u(x,t)\geq1-k=\underline{u}_1\mbox{ uniformly in any compact subset of }[0,\infty).\lbl{it4}\eess
This completes the proof of Step 1.

{\bf Step 2}. If \eqref{it} holds for $n=j\geq 1$, then it holds for $n=j+1$.

Since \eqref{it} holds for $n=j$,  for any small $\epsilon>0$ and large  $l>\max\left\{g_0,\frac{\pi}{2}\frac{1}{\sqrt{1-k}},\frac{\pi}{2}\sqrt{\frac{d}{r(1-h)}}\right\}$, there is $t_2>0$ such that
\[
g(t)>l,\; u(x,t)\in[\underline{u}_j-\epsilon, \overline u_j+\epsilon], \;v(x,t)\in [\underline{v}_j-\epsilon, \overline v_j+\epsilon]
\;\mbox{ for } x\in [0, l],\;  t>t_2.
\]

It follows from the comparison principle that $u(x,t)\leq\overline{u}(x, t)$ for $x\in[0, l]$ and $t>t_2$, where $\overline{u}(x,t)$ satisfies
\bess
\left\{
\begin{aligned}
&\overline{u}_t-\overline u_{xx}=\overline{u}(1-\overline{u}-k(\underline{v}_j-\epsilon)),&&x\in (0,l),&t>t_2,\\
&\overline{u}_x(0,t)=0,\; \overline u(l, t)=\overline u_j+\epsilon, &&& t>t_2,\\
&\overline u(x,t_2)=u(x, t_2), && x\in [0, l].
\end{aligned}
\right.
\eess
It is well known that this problem has a unique positive steady-state solution $\hat u^*(x)$ and
$\lim_{t\rightarrow\infty}\hat{u}(x,t)=\hat u^*(x)$ uniformly for $x\in [0,l]$. Moreover,
\[
\lim_{l\to\infty} \hat u^*(x)=1-k(\underline v_j-\epsilon) \mbox{ locally uniformly in } [0,\infty).
\]
It follows, since $\epsilon>0$ can be arbitrarily small,  that
\[
\limsup_{t\to\infty} u(x,t)\leq 1-k\underline v_j=\overline u_{j+1}  \mbox{ locally uniformly in } [0,\infty).
\]

Analogously, from the comparison principle we obtain $u(x,t)\geq\underline{u}(x, t)$ for $x\in[0, l]$ and $t>t_2$, where $\underline{u}(x,t)$ satisfies
\bess
\left\{
\begin{aligned}
&\underline{u}_t-\underline u_{xx}=\underline{u}(1-\underline{u}-k(\overline{v}_j+\epsilon)),&&x\in (0,l),&t>t_2,\\
&\overline{u}_x(0,t)=0,\; \overline u(l, t)=\underline u_j-\epsilon, &&& t>t_2,\\
&\overline u(x,t_2)=u(x, t_2), && x\in [0, l],
\end{aligned}
\right.
\eess
from which we can deduce
\[
\liminf_{t\to\infty} u(x,t)\geq 1-k\overline v_j=\underline u_{j+1}  \mbox{ locally uniformly in } [0,\infty).
\]

The proof for
\[
\limsup_{t\to\infty} v(x,t)\leq \overline v_{j+1}, \; \liminf_{t\to\infty} v(x,t)\geq \underline v_{j+1}  \mbox{ locally uniformly in } [0,\infty)
\]
is similar, and we omit the details.
\end{proof}

Theorem \ref{f2} now follows directly from Lemmas \ref{t4} and \ref{t6}.

\begin{lem}\lbl{t5} If $g_\infty<+\infty$, then $g_\infty\leq \frac{\pi}{2\sqrt{1-k}}$. Hence $g_0\geq \frac{\pi}{2\sqrt{1-k}}$ implies $g_\infty=+\infty$.
\end{lem}

\begin{proof}\ \ Assume for contradiction that $\frac{\pi}{2\sqrt{1-k}}<g_\infty<\infty$. Then there exists $T_1>0$ such that $g(T_1)>\frac{\pi}{2\sqrt{1-k(1+\epsilon)}}$ for $\epsilon$ sufficiently small. By a simple comparison consideration,
 there exists $T>T_1>0$ such that
 $$v(x,t)\leq1+\epsilon\mbox{ for }x\in[0,\infty), \quad t>T.$$
Hence $(u,g)$ satisfies
\bess
\left\{
\begin{aligned}
&u_t\geq u_{xx}+u(1-u-k(1+\epsilon)),&0<x<g(t),&\quad t>T,\\
&u_x(0,t)=u(g(t),t)=0,&&\quad t>T,\\
&g'(t)=-\gamma u_x(g(t),t),&&\quad t>T,\\
&u(x,T)>0,&0<x<g(T),&
\end{aligned}
\right.
\eess
which implies that $(u,g)$ is an upper solution to the problem
\bess
\left\{
\begin{aligned}
&w_t=w_{xx}+w(1-w-k(1+\epsilon)),&0<x<\underline{g}(t),&\quad t>T,\\
&w_x(0,t)=w(\underline{g}(t),t)=0,&&\quad t>T,\\
&\underline{g}'(t)=-\gamma w_x(\underline{g}(t),t),&&\quad t>T,\\
&w(x,T)=u(x,T),\quad \underline{g}(T)=g(T),&0<x<\underline{g}(T).&
\end{aligned}
\right.
\eess
Thus, $g(t)\geq\underline{g}(t)$ for $t>T$. Since $\underline{g}(T)=g(T)>g(T_1)>\frac{\pi}{2\sqrt{1-k(1+\epsilon)}}$, it follows from Lemma \ref{t3} that $\underline{g}(t)\rightarrow\infty$ and hence $g_\infty=\infty$. This contradiction leads to $g_\infty<\frac{\pi}{2\sqrt{1-k}}$.\end{proof}

\begin{lem}\label{t7} If $g_0<\frac{\pi}{2\sqrt{1-k}}$, then there exists $\underline{\gamma}\geq0$ depending on $u_0$ and $v_0$ such that spreading happens when $\gamma>\underline{\gamma}$.
\end{lem}

\begin{proof}Since $\limsup_{t\rightarrow\infty}v(x,t)\leq1$ uniformly for $x\in[0,\infty)$, there exists $t_3>0$, which is independent of $\gamma$, such that $v(x,t)\leq 1+\epsilon$ for $x\in[0,\infty),\; t\geq t_3$. Thus $(u,g)$ satisfies
\bess\left\{
\begin{aligned}
&u_t-u_{xx}\geq u(1-u-k(1+\epsilon)),&0<x<g(t),&\quad t>t_3,\\
&u_x(0,t)=0,\quad u(g(t),t)=0,&&\quad t>t_3,\\
&g'(t)=-\gamma u_x(g(t),t),&&\quad t>t_3,\\
&u(x,t_3)>0,&0\leq x<g(t_3).&\quad
\end{aligned}
\right.\lbl{}
\eess
Hence
 $(u,g)$ is an upper solution to the problem
\bes\left\{
\begin{aligned}
&\hat{u}_t-\hat{u}_{xx}=\hat{u}(1-\hat{u}-k(1+\epsilon)),&0<x<\hat{g}(t),&\quad t>t_3,\\
&\hat{u}_x(0,t)=0,\quad \hat{u}(\hat{g}(t),t)=0,&&\quad t>t_3\\
&\hat{g}'(t)=-\gamma \hat{u}_x(\hat{g}(t),t),&&\quad t>t_3,\\
&\hat{u}(x,t_3)=u(x,t_3),\quad \hat g(t_3)=g(t_3),&0\leq x<g(t_3).&
\end{aligned}
\right.\lbl{s3-50}
\ees
The comparison principle infers $g(t)\geq \hat g(t)$ for $t>t_3$.
Applying Lemma \ref{t3} to \eqref{s3-50} we see that there exists $\underline{\gamma}\geq0$ depending on $g(t_3)$ and $u(x, t_3)$ (which are uniquely determined by $u_0$ and $v_0$) such that spreading happens for \eqref{s3-50} when $\gamma>\underline{\gamma}$.
Thus $\lim_{t\to\infty} g(t)=\infty$ when $\gamma>\underline \gamma$, and by Lemma \ref{t6}, spreading happens to \eqref{section1-1} for such $\gamma$.
\end{proof}

\begin{lem}\label{t8} There exists $\gamma^\ast\geq0$, depending on $u_0$ and $v_0$, such that $g_\infty<\infty$ if $\gamma\leq\gamma^\ast$, and $g_\infty=\infty$ if $\gamma>\gamma^\ast$.
\end{lem}
\begin{proof}
Set $\Lambda=\{\gamma>0:g_\infty>\frac{\pi}{2\sqrt{1-k}}\}.$
It follows from Lemmas \ref{t5} and \ref{t7} that $\gamma^\ast:=\inf\Lambda\in[0,\infty)$.
The comparison principle infers that
 $g_\infty=\infty$ if $\gamma>\gamma^\ast$ and $g_\infty<\infty$ if $0<\gamma<\gamma^\ast$.

It remains to show that $\gamma^\ast\not\in\Lambda$. Otherwise, $\gamma^*>0$ and $g_{\infty}>\frac{\pi}{2\sqrt{1-k}}$ for $\gamma=\gamma^\ast$. Hence we can find $T > 0$ such that $g(T)>\frac{\pi}{2\sqrt{1-k}}$. To emphasize the dependence of the solution of \eqref{section1-1} on $\gamma$, we denote it by $(u_\gamma,v_\gamma,g_\gamma)$ instead of $(u,v,g)$, and so $g_{\gamma^\ast}(T)>\frac{\pi}{2\sqrt{1-k}}$. By the
continuous dependence of $(u_\gamma,v_\gamma,g_\gamma)$ on $\gamma$, we can find $\epsilon>0$ small so that $g_\gamma(T)>\frac{\pi}{2\sqrt{1-k}}$ for $\gamma\in[\gamma^\ast-\epsilon,\gamma^\ast+\epsilon]$. It then follows
from Lemma \ref{t5} that for all such $\gamma$, $\lim_{t\rightarrow\infty}g_\gamma(t)=\infty$. This implies that $[\gamma^\ast-\epsilon,\gamma^\ast+\epsilon]\subset\Lambda$, and $\inf\Lambda\leq\gamma^\ast-\epsilon$, a contradiction to the definition of $\gamma^\ast$.
\end{proof}

\begin{lem}\label{t9}If $g_0<\frac{\pi}{2}$, then there exists $\overline{\gamma}>0$ depending on $u_0$ such that $g_\infty<+\infty$ if $\gamma\leq\overline{\gamma}$.
\end{lem}
\begin{proof}Clearly, $(u,g)$ satisfies
\bess\left\{
\begin{aligned}
&u_t-u_{xx}\leq u(1-u),&0<x<g(t),&\quad t>0,\\
&u_x(0,t)=0,\quad u(g(t),t)=0,&&\quad t>0,\\
&g'(t)=-\gamma u_x(g(t),t),&&\quad t>0,\\
&u(x,0)=u_0(x),&0\leq x\leq g_0.&
\end{aligned}
\right.\lbl{(se111)}
\eess
That is, $(u,g)$ is a lower solution to the problem
\bess\left\{\begin{aligned}
&\bar{u}_t-\bar{u}_{xx}= \bar{u}(1-\bar{u}),&0<x<\bar{g}(t),&\quad t>0,\\
&\bar{u}_x(0,t)=0,\quad\bar{u}(\bar g(t),t)=0,&&\quad t>0,\\
&\bar{g}'(t)=-\gamma \bar{u}_x(\bar{g}(t),t),&&\quad t>0,\\
&\bar{u}(x,0)=u_0(x),&0\leq x\leq g_0.&
\end{aligned}
\right.\lbl{(se111)}
\eess
It follows from the comparison principle that $g(t)\leq\bar{g}(t)$. Since $g_0<\pi/2$, by Lemma \ref{t3}  there exists $\bar{\gamma}>0$ depending on $u_0$ such that $\bar{g}(\infty)<\infty$ if $\gamma\leq\bar{\gamma}$. Hence, $g_\infty<+\infty$ if $\gamma\leq\overline{\gamma}$.
\end{proof}

Theorem \ref{f3} now follows directly from Lemmas \ref{t5}-\ref{t9}.

{\setlength{\baselineskip}{16pt}{\setlength\arraycolsep{2pt} \indent
\section{Asymptotic spreading speed}
\setcounter{equation}{0}

We prove Theorem 1.3 in this section.

\begin{lem}\lbl{T15}Suppose spreading occurs, i.e., alternative {\rm (i)} happens in Theorem 1.1. Then
\bess\liminf_{t\rightarrow\infty}\frac{g(t)}{t}\geq c_\gamma.\lbl{section4e}\eess
\end{lem}

\begin{proof}\, Let $V(t)$ be the unique solution of $$V'=rV(1-V),\quad V(0)=\|v_0\|_\infty.$$
Then a simple comparison consideration yields $v(x,t)\leq V(t)$ for $x\geq0$ and $t>0$. Since $\lim_{t\rightarrow\infty}V(t)=1$, we can find $T_0'>0$ such that
\bes\lbl{infe1}
v(x,t)<1+\delta\quad\mbox{for } x\geq0,\quad t\geq T_0'.
\ees

We now consider the auxiliary problem
\bes\left\{
\begin{aligned}
&\phi_\delta''-c\phi_\delta'+\phi_\delta(1-2\delta-\phi_\delta-k\psi_\delta)=0,\quad \phi_\delta'>0\quad\mbox{for } s>0,\\
&d\psi_\delta''-c\psi_\delta'+r\psi_\delta(1+2\delta-\psi_\delta-h\phi_\delta)=0,\quad \psi_\delta'<0\quad\mbox{for } -\infty<s<\infty,\\
&\phi_\delta(s)=0\quad \mbox{for } s\leq 0,\; \gamma \phi_\delta'(0)=c,\quad \psi_\delta(-\infty)=1+2\delta,\\
&(\phi_\delta(\infty),\psi_\delta(\infty))=(u^*_\delta, v^*_\delta):=
\left(\frac{1-2\delta-k(1+2\delta)}{1-hk},\frac{1+2\delta-h(1-2\delta)}{1-hk}\right),
\end{aligned}
\right.\lbl{section401}\ees
where $\delta>0$ is small. We claim that
 there exists a unique $c^\delta_{\gamma}>0$ such that \eqref{section401} has a unique solution $(\phi_\delta,\psi_\delta)$ when $c=c_\gamma^\delta$; moreover,
$$
 \lim_{\delta\rightarrow0}c^\delta_{\gamma}=c_\gamma.
$$

Indeed, if we define
\[
\sigma_{\pm\delta}=\sqrt{1\pm 2\delta},\;  \phi_\delta(s)=\sigma_{-\delta}^2\tilde \phi_\delta(\sigma_{-\delta}s),
\;
\psi_\delta(s)=\sigma_{+\delta}^2\tilde \psi_\delta(\sigma_{-\delta}s),
\]
and
\[
\tilde c=c/\sigma_{-\delta},\; \tilde\gamma=\gamma \sigma_{-\delta}^4,
\;
\tilde r_\delta=r \left(\frac{\sigma_{+\delta}}{\sigma_{-\delta}}\right)^2,\; \tilde h_\delta=h \left(\frac{\sigma_{-\delta}}{\sigma_{+\delta}}\right)^2,\; \tilde k_\delta=k \left(\frac{\sigma_{+\delta}}{\sigma_{-\delta}}\right)^2,
\]
then a direct calculation shows that $(c, \phi_\delta, \psi_\delta)$ solves \eqref{section401} if and only if $(\tilde c, \tilde \phi_\delta,\tilde \psi_\delta)$ satisfies \eqref{section12} and $\tilde\gamma \tilde\phi'(0)=\tilde c$ when
$(r,h,k, u^*, v^*)$  in \eqref{section12} is replaced by $(\tilde r_\delta,\tilde h_\delta,\tilde k_\delta, u^*_\delta, v^*_\delta)$. So the claim follows directly from
Theorem \ref{F4} and Lemma \ref{tlem13}, and the continuous dependence of the unique solution on the parameters.

Since $\psi_\delta(-\infty)=1+2\delta$, $\psi_\delta(+\infty)=v^*_\delta>v^*$ and $\psi_\delta'<0$, there exists $L>1$ such that
\bes\psi_\delta(x)>1+\delta \mbox{ for } x\leq - L,\;
\psi_\delta(x)>v^*_\delta>v^* \mbox{ for }x\geq -L.
\lbl{infe5}\ees
Similarly it follows from $\phi_\delta(+\infty)=u^*_\delta<u^*$ and $\phi_\delta'>0$ that
\[
\phi_\delta(x)<u^*_\delta<u^* \mbox{ for } x\geq 0.
\]

By the spreading assumption, we have
\bess
\lim_{t\to\infty} g(t)=\infty,\; \lim_{t\to\infty}(u(x,t),v(x,t))= (u^*, v^*) \mbox{ in any compact subset of } [0,\infty).
\eess
Hence, in view of $u^*>u^*_\delta$ and $v^*<v^*_\delta$, there exists $T_0>T'_0$ large such that for $x\in [0, L+1]$ and $t\geq T_0$,
\[
g(t)>1,\; u(x,t)>u^*_\delta,\; v(x,t)<v^*_\delta.
\]

We now define
\bess\begin{aligned}
&\underline{g}(t)=c^\delta_{\gamma} (t-T_0)+1,\ \
&\underline{u}(x,t)=\phi_\delta(\underline{g}(t)-x),\ \
&\overline{v}(x,t)=\psi_\delta(\underline{g}(t)-x).
\end{aligned}\eess
Then $\underline g(T_0)=1<g(T_0)$,
\[
 \underline u(x, T_0)=\phi_\delta(1-x)<u^*_\delta<u(x, T_0) \mbox{ for } x\in [0,1]=[0,\underline g(T_0)],
\]
and in view of \eqref{infe5} and \eqref{infe1}, we also have
\[\begin{aligned}
&\overline v(x, T_0)=\psi_\delta(1-x)>v^*_\delta>v(x, T_0) \mbox{ for } x\in [0, L+1],
\\
&\overline v(x, T_0)=\psi_\delta(1-x)>1+\delta>v(x, T_0) \mbox{ for } x>L+1.
\end{aligned}
\]
Let us also note that
\[
\underline u(0, t)=\phi_\delta(\underline g(t))<u^*_\delta<u(0, t) \mbox{ for } t\geq T_0,
\]
\[
\overline v(0, t)=\psi_\delta(\underline g(t))>v^*_\delta>u(0, t) \mbox{ for } t\geq T_0,
\]
and moreover,
\bess
\underline{g}'(t)=c^\delta_{\gamma}=\gamma\phi_\delta'(0)=-\gamma\underline{u}_x(\underline{g}(t),t) \mbox{ for } t\geq T_0.
\eess
Furthermore,
\bess\begin{aligned}
\underline{u}_t-\underline{u}_{xx}&=c^\delta_{\gamma}\phi_\delta'-\phi_\delta''\\
&=\phi_\delta(1-2\delta-\phi_\delta-k\psi_\delta)\\
&\leq\phi_\delta(1-\phi_\delta-k\psi_\delta)\\
&=\underline{u}(1-\underline{u}-k\overline{v}),
\end{aligned}\eess
\bess\begin{aligned}
\overline{v}_t-d\overline{v}_{xx}&=c^\delta_{\gamma}\psi_\delta'-d\psi_\delta''\\
&=r\psi_\delta(1+2\delta-\psi_\delta-h\phi_\delta)\\
&\geq r\psi_\delta(1-\psi_\delta-h\phi_\delta)\\
&= r\overline{v}(1-\overline{v}-h\underline{u}).
\end{aligned}\eess
Hence, we can use Proposition \ref{the comparison principle} and Remark \ref{remarkp} to conclude that
$$\begin{array}{l}
g(t)\geq\underline{g}(t) \mbox{ for } t\geq T_0.
\end{array}$$
It follows that
$\liminf_{t\rightarrow\infty}\frac{g(t)}{t}\geq c^\delta_\gamma$, which yields the required inequality by letting $\delta\rightarrow0$.
\end{proof}

\begin{lem}\lbl{T14} Under the assumptions of Lemma \ref{T15} we have
\bess\limsup_{t\rightarrow\infty}\frac{g(t)}{t}\leq c_\gamma.\lbl{section4e}\eess
\end{lem}

\begin{proof}\, For small $\tau>0$ we consider the auxiliary problem
\bes\left\{
\begin{aligned}
&\phi_\tau''-c\phi_\tau'+\phi_\tau(1+2\tau-\phi_\tau-k\psi_\tau)=0,\quad \phi_\tau'>0\quad \mbox{for }  0<s<\infty,\\
&d\psi_\tau''-c\psi_\tau'+r\psi_\tau(1-2\tau-\psi_\tau-h\phi_\tau)=0,\quad \psi_\tau'<0 \quad \mbox{for }-\infty<s<\infty,\\
&\phi_\tau(s)\equiv0\quad\mbox{for } s\leq0,\; \gamma \phi'_\tau(0)=c,\quad \psi_\tau(-\infty)=1-2\tau,\\
&(\phi_\tau,\psi_\tau)(\infty)=(u^*_\tau, v^*_\tau):=\left(\frac{1+2\tau-k(1-2\tau)}{1-hk},\frac{1-2\tau-h(1+2\tau)}{1-hk}\right).
\end{aligned}
\right.\lbl{teif}\ees
As in the proof of Lemma \ref{T15} we can use a change of variable trick to reduce \eqref{teif} to \eqref{section12}, and
then apply Theorem \ref{F4} and Lemma \ref{tlem13} to conclude that there exists a unique $c^\tau_{\gamma}>0$ such that \eqref{teif} has a unique solution $(\phi_\tau,\psi_\tau)$ when $c=c^\tau_\gamma$, and
moreover,
 $$ \lim_{\tau\rightarrow0}c^\tau_{\gamma}=c_\gamma.$$
Let us also observe that
\[
u^*<u^*_\tau,\; v^*>v^*_\tau,\; 0<\phi_\tau(x)<u^*_\tau \mbox{ for } x>0,\; v^*_\tau<\psi_\tau(x)<1-2\tau \mbox{ for } x\in (-\infty, \infty).
\]
For clarity we divide the analysis below into three steps.

{\bf Step 1.}
We prove that for any small $\tau>0$, we can find $T'_0>0$ such that for each $T\geq T'_0$, there exists $L(T)>0$
having the following property:
$$
v(x,T)\geq1-\tau\mbox{ for } x\geq L(T).
$$

By \eqref{section1-1a} we have
$$\tilde v_0:=\inf_{x\geq 0}v_0(x)>0.$$
Consider the auxiliary problem
\bes\lbl{tif}
\left\{
\begin{aligned}
&w_t-dw_{xx}=rw(1-w),&x>0,&\quad t>0,\\
&w(0,t)=0,&&\quad t>0,\\
&w(x,0)=\tilde v_0,&x>0.&
\end{aligned}
\right.
\ees
It is well known that the solution of \eqref{tif} satisfies
$$
\lim_{t\rightarrow\infty}w(x,t)=w_\ast(x)\mbox{ locally unformly for }x\in[0,\infty),
$$
where $w_\ast$ is the unique solution of
$$
-w_\ast''=rw_\ast(1-w_\ast)\mbox{ for }x\in[0,\infty),w_\ast(0)=0.
$$
Moreover, $w_*$ has the property that $w_\ast'>0$ and $w_\ast(\infty)=1$. Therefore, there exist positive constants $L_1, T'_0$ large enough such that
$$
w(L_1,t)\geq w_\ast(L_1)-\tau/2\geq1-\tau\mbox{ for }t\geq T'_0.
$$
Applying the maximum principle to the equation satisfied by $w_x(x,t)$, we deduce $w_x(x,t)\geq0$ for $x>0$ and $t>0$. It follows that
$$
w(x,t)\geq1-\tau\mbox{ for }x\geq L_1\mbox{ and }t\geq T'_0.
$$
Fix $T\geq T_0'$ and note that $v$ satisfies
$$\begin{cases}
v_t-dv_{xx}=rv(1-v), & x>g(T),\; 0<t\leq T, \\
v(x,0)\geq \tilde v_0, & x\geq g(T).\end{cases}$$
Set $\tilde{w}(x,t):=w(x-g(T),t).$
Then $\tilde{w}(x,t)$ satisfies
$$
\tilde{w}_t-d\tilde{w}_{xx}=r\tilde{w}(1-\tilde{w})\mbox{ for }x>g(T) \mbox{ and }0<t\leq T.
$$
Since
$$
\tilde{w}(g(T), t)=0<v(g(T), t)\mbox{ for } t\in (0, T],\quad \tilde{w}(x, 0)=\tilde v_0\leq v(x,0)\mbox{ for }x>g(T),
$$
we can use the comparison principle to deduce
$$
v(x,t)\geq\tilde{w}(x,t)=w(x-g(T),t)\mbox{ for }x>g(T)\mbox{ and }0<t\leq T .
$$
Thus we obtain
$$v(x,T)\geq w(x-g(T),T)\geq w(L_1,T)\geq1-\tau\mbox{ for }x\geq L(T):=L_1+g(T).$$
 This completes the proof of Step 1.

{\bf Step 2.} We prove that for any small $\tau>0$, there exists $T_1'>0$ such that
\bes\lbl{comp-ode}
u(x,t)\leq u^*_{\tau/2},\; v(x, t)\geq v^*_{\tau/2} \mbox{ for } x\geq 0,\; t\geq T_1'.
\ees

We prove the claimed inequalities in \eqref{comp-ode} by a comparison argument involving the following ODE system
\[\begin{cases}
\check u'(t)=\check  u(1-\check  u-k\check  v),& t>0,\\
\check v'(t)=r\check  v(1-\check  v-h\check  u), & t>0,\\
(\check  u(0), \check  v(0))=(\|u_0\|_\infty, \tilde v_0).&
\end{cases}
\]
Indeed, by the comparison principle for coorporative system we easily obtain
\[
u(x, t)\leq \check  u(t),\;  v(x,t)\geq \check  v(t) \mbox{ for } x\geq 0,\; t>0.
\]
But it is well known (for example, see \cite{Goh}) that 
\[
\lim_{t\to\infty} (\check  u(t), \check  v(t))=(u^*, v^*).
\]
The inequalties in \eqref{comp-ode} thus follow directly once we recall $u^*<u^*_{\tau/2}<u^*_\tau$ and $v^*>v^*_{\tau/2}>v^*_\tau$.

{\bf Step 3.} We complete the proof of the lemma by
constructing a suitable comparison function triple $(\overline{u}(x,t),
\underline{v}(x,t),\overline{g}(t))$, and applying the comparison principle.

We fix $T_0:=\max\{T_0', T_1'\}$. Then by the conclusions in Steps 1 and 2 we obtain
\[
u(x,T_0)\leq u^*_{\tau/2}<u^*_\tau \mbox{ for } x\geq 0,\; v(x, T_0)\geq v^*_{\tau/2}>v^*_\tau \mbox{ for } x\geq 0,
\]
and
\[
v(x, T_0)\geq 1-\tau \mbox{ for } x\geq L(T_0).
\]
Choose $S>L(T_0)>g(T_0)$ large so that
\[
\phi_\tau(x)>u^*_{\tau/2},\; \psi_\tau(x)<v^*_{\tau/2} \mbox{ for } x\geq S-L(T_0),
\]
and then define
\bess\begin{aligned}
&\overline{g}(t)=c^\tau_{\gamma} (t-T_0)+S,\ \
&\overline{u}(x,t)=\phi_\tau(\overline{g}(t)-x),\ \
&\underline{v}(x,t)=\psi_\tau(\overline{g}(t)-x).
\end{aligned}\eess
Clearly $\overline{g}(T_0)=S> g(T_0)$ and
\bess
\overline{g}'(t)=c^\tau_{\gamma}=\gamma\phi_\tau'(0)=-\gamma\overline{u}_x(\overline{g}(t),t)\quad \mbox{ for }
t\geq T_0.
\eess
Moreover,
\bess\begin{aligned}
&\overline{u}(x,T_0)=\phi_\tau(S-x)\geq\phi_\tau(S-g(T_0))>u^*_{\tau/2}\geq u(x,T_0)\quad\mbox{ for } x\in[0,g(T_0)],\\
&\underline{v}(x,T_0)=\psi_\tau(S-x)\leq\psi_\tau(S-L(T_0))<v^*_{\tau/2}\leq v(x,T_0)\quad\mbox{for } x\in[0,L(T_0)],
\end{aligned}\eess
and for $x>L(T_0)$,
$$\underline{v}(x,T_0)=\psi_\tau(S-x)<\psi_\tau(-\infty)=1-2\tau<v(x,T_0).$$
Furthermore,
\bess\begin{aligned}
&\overline{u}_x(0,t)=-\phi_\tau'(\overline{g}(t))<0,\ \
\overline{u}(\overline{g}(t),t)=\phi_\tau(0)=0,\ \
\underline{v}_x(0,t)=-\psi_\tau'(\overline{g}(t))>0 \quad \mbox{ for } t\geq T_0.
\end{aligned}\eess
Finally, direct calculations show that
\bess\begin{aligned}
\overline{u}_t-\overline{u}_{xx}&=c^\tau_{\gamma}\phi_\tau'-\phi_\tau''\\
&=\phi_\tau(1+2\tau-\phi_\tau-k\psi_\tau)\\
&\geq\phi_\tau(1-\phi_\tau-k\psi_\tau)\\
&=\overline{u}(1-\overline{u}-k\underline{v})
\end{aligned}\eess
and
\bess\begin{aligned}
\underline{v}_t-d\underline{v}_{xx}&=c^\tau_{\gamma}\psi_\tau'-d\psi_\tau''\\
&=r\psi_\tau(1-2\tau-\psi_\tau-h\phi_\tau)\\
&\leq r\psi_\tau(1-\psi_\tau-h\phi_\tau)\\
&=r\underline{v}(1-\underline{v}-h\overline{u}).
\end{aligned}\eess
Hence, we can use Proposition \ref{the comparison principle} to conclude that
$$
g(t)\leq\overline{g}(t)\quad \mbox{ for } t\geq T_0.
$$
It follows that
$\limsup_{t\rightarrow\infty}\frac{g(t)}{t}\leq c^\tau_\gamma$, which gives the required inequality  by letting $\tau\rightarrow0$.
\end{proof}

Theorem \ref{F5} now follows directly from Lemmas \ref{T15} and \ref{T14}. So the asymptotic spreading speed of $u$ in Theorem \ref{F5} is given by $c^0=c_\gamma$.

\section{Proof of Proposition \ref{s7}}
\setcounter{equation}{0}
Although we follow some standard steps in the proof of Proposition \ref{s7}, since the first equation of \eqref{function1}
is only satisfied for $s>0$, nontrivial changes are needed. We break the rather long proof  into several lemmas.

We start with a second order ODE of the following form
\bes\left\{\begin{aligned}
&d_1y''-cy'-\beta y+f(s)=0,\quad s>0,\\
& y(0)=0,
\end{aligned}\right.\lbl{2.3}
\ees
where the constants $c$ and $\beta$ are positive, and the nonlinear function $f$ is specified below.

Let $$
\lambda_{1}=\frac{c-\sqrt{c^2+4\beta d_1}}{2d_1},\quad\lambda_{2}=\frac{c+\sqrt{c^2+4\beta d_1}}{2d_1}
$$
be the two roots of equation $d_1\lambda^2-c\lambda-\beta=0.$
Then we have the following result.
\begin{lem}\lbl{s1}Assume $f:[0,\infty)\rightarrow\mathbb{R}$ is piecewise continuous and $|f(s)|\leq Ae^{\alpha s}$  for all $s\geq 0$ and some constants $A>0$, $\alpha\in(0,\min\{-\lambda_1,\lambda_2\})$.
Then  \eqref{2.3} has a unique solution satisfying $y(s)=O(e^{\alpha s})$ as $s\rightarrow\infty$, and it is given by
\bes
 y(s)=\frac{1}{d_1(\lambda_{2}-\lambda_{1})}\left[\int_{0}^sK_1(\xi,s)f(\xi)d\xi
+\int_s^{\infty}K_2(\xi,s)f(\xi)d\xi\right],
\lbl{lem1f0}
\ees
where
\bess
K_1(\xi,s)=e^{\lambda_1s}\big(e^{-\lambda_1\xi}-e^{-\lambda_{2}\xi}\big),\quad
K_2(\xi,s)=\big(e^{\lambda_{2}s}-e^{\lambda_{1}s}\big)e^{-\lambda_2\xi}.
\lbl{lem1f01}
\eess
\end{lem}

\begin{proof}\, By the variation of constants formula, the solutions of \eqref{2.3} are given by
\bes\lbl{lem1f1}
 \;\;\;\; y(s)=\gamma\big (e^{\lambda_{1}s}-e^{\lambda_{2}s}\big)+\frac{1}{d_1(\lambda_{2}-\lambda_{1})}
\left[\int_{0}^se^{\lambda_{1}(s-\xi)}f(\xi)d\xi
-\int_{0}^se^{\lambda_{2}(s-\xi)}f(\xi)d\xi\right],\;\gamma\in\mathbb R.
\ees
Multiplying both sides of \eqref{lem1f1} by $e^{-\lambda_{2}s}$, we get
\bess\displaystyle
 y(s)e^{-\lambda_{2}s}=\gamma\big(e^{(\lambda_{1}-\lambda_{2})s}-1\big)+
\frac{e^{-\lambda_{2}s}}{d_1(\lambda_{2}-\lambda_{1})}
\left[\int_{0}^se^{\lambda_{1}(s-\xi)}f(\xi)d\xi
-\int_{0}^se^{\lambda_{2}(s-\xi)}f(\xi)d\xi\right].\lbl{lem1f11}
\eess
If $ y(s)=O(e^{\alpha s})$ as $s\rightarrow\infty$, then due to $\lambda_{1}<0<\lambda_{2}$ and $|\lambda_{1}|<\lambda_{2}$, we obtain
\bess
 y(s)e^{-\lambda_{2}s}\rightarrow0,\quad e^{(\lambda_{1}-\lambda_{2})s}\rightarrow0
\mbox{ and }\frac{e^{-\lambda_{2}s}}{d_1(\lambda_{2}-\lambda_{1})}
\int_{0}^se^{\lambda_{1}(s-\xi)}f(\xi)d\xi\rightarrow0
\eess
as $s\rightarrow\infty$. Therefore,
\bes
\gamma=\frac{-1}{d_1(\lambda_{2}-\lambda_{1})}
\int_{0}^\infty e^{-\lambda_{2}\xi}f(\xi)d\xi.
\lbl{lem1fu2}
\ees
Substituting \eqref{lem1fu2} into  \eqref{lem1f1}, we obtain  \eqref{lem1f0}.

If $ y(s)$ is given by \eqref{lem1f0}, then it is easy to check that $ y(s)$ satisfies  \eqref{2.3} and $ y(s)=O(e^{\alpha s})$ as $s\rightarrow\infty$.
\end{proof}

Define the operators $H_1:C_{\mathcal{R}}(\mathbb{R}^+,\mathbb{R}^2)\rightarrow C(\mathbb{R^+},\mathbb{R})$ and $H_2:C_{\mathcal{R}}(\mathbb{R},\mathbb{R}^2)\rightarrow C(\mathbb{R},\mathbb{R})$ by
\bess
\begin{aligned}
&H_1(\varphi)(s):=\beta \varphi_1(s)+f_1(\varphi(s)),\\
&H_2(\varphi)(s):=\beta \varphi_2(s)+f_2(\varphi(s)),
\end{aligned}
\eess
where the positive constant $\beta$ is  large enough such that $H_i(\varphi)$ is nondecreasing with respect to $\varphi_1$ and $\varphi_2$, for $(\varphi_1(s),\varphi_2(s))\in \mathcal R=[0, k_1]\times [0,k_2]$.

Let $F_1:C_{\mathcal{R}}(\mathbb{R},\mathbb{R})\rightarrow C(\mathbb{R},\mathbb{R})$ be given by
\bes
F_1(\varphi)(s):=\left\{\begin{aligned}
&\frac{1}{d_1(\lambda_{2}-\lambda_{1})}\left[\int_{0}^sK_1(\xi,s)H_1(\varphi)(\xi)d\xi\right. &\\
&\quad\quad\hspace{1.5cm}+\left.\int_s^{\infty}K_2(\xi,s)H_1(\varphi)(\xi)d\xi\right],&s>0,\\
&0,&s\leq0,
\end{aligned}\right.\lbl{l01}
\ees
where $K_i(\xi,s)$ is given by \eqref{lem1f01}. By Lemma \ref{s1},
it is easy to see that the operator $F_1$ is well defined and
\bess
\left\{\begin{aligned}
&d_1(F_1(\varphi))''(s)-c(F_1(\varphi))'(s)-\beta F_1(\varphi)(s)+H_1(\varphi)(s)=0,&s>0,\\
&F_1(\varphi)(s)=0,&s\leq0.
\end{aligned}\right.\lbl{slt1}
\eess

Let $$\mu_{1}=\frac{c-\sqrt{c^2+4\beta d_2}}{2d_2},\quad\mu_{2}=\frac{c+\sqrt{c^2+4\beta d_2}}{2d_2}$$
be the two roots of
$$d_2\mu^2-c\mu-\beta=0.$$
Define  $F_2:C_{\mathcal{R}}(\mathbb{R},\mathbb{R})\rightarrow C(\mathbb{R},\mathbb{R})$ by
\bes\lbl{l02}\begin{aligned}
F_2(\varphi)(s):&=\frac{1}{d_2(\mu_2-\mu_1)}\left[\int_{-\infty}^se^{\mu_1(s-\xi)}H_2(\varphi)(\xi)d\xi +\int_s^{\infty}e^{\mu_2(s-\xi)}H_2(\varphi)(\xi)d\xi\right].
\end{aligned}\ees
It is easy to show that the operator $F_2$ is well defined and satisfies
\bess
d_2(F_2(\varphi))''(s)-c(F_2(\varphi))'(s)-\beta F_2(\varphi)(s)+H_2(\varphi)(s)=0.\lbl{slt2}
\eess

We now define
 $F: C_{\mathcal{R}}(\mathbb{R},\mathbb{R}^2)\rightarrow C(\mathbb{R},\mathbb{R}^2)$ by
\[
F(\varphi):=(F_1(\varphi), F_2(\varphi)).
\]
 Clearly, $\varphi$ is a fixed point of the operator $F$ in $C_{\mathcal{R}}(\mathbb{R},\mathbb{R}^2)$ if and only if it is a solution of \eqref{function1} in $C_{\mathcal{R}}(\mathbb{R},\mathbb{R}^2)$.

Next, we introduce a Banach space with exponential decay norm. Fix $\sigma\in(0,\min\{|\lambda_{1}|,\lambda_{2},|\mu_{1}|,\mu_{2}\})$. It is easy to see that
$$B_\sigma(\mathbb{R},\mathbb{R}^2):=\left\{\varphi\in C(\mathbb{R},\mathbb{R}^2):\sup_{s\in\mathbb{R}}|\varphi(s)|e^{-\sigma |s|}<\infty\right\}$$
equipped with the  norm
$$|\varphi|_\sigma:=\sup_{s\in\mathbb{R}}|\varphi(s)|e^{-\sigma |s|}$$
 is a Banach space.

Let $\overline{\varphi}(s)$ and  $\underline{\varphi}(s)$ be the upper and lower solutions given in the statement of Proposition \ref{s7}.
Consider the set
$$\Gamma:=\Big\{\varphi=(\varphi_1,\varphi_2)\in B_\sigma(\mathbb{R},\mathbb{R}^2): \underline{\varphi}\leq\varphi\leq\overline{\varphi},\;
 \varphi_i\mbox{ is nondecreasing for } s\in\mathbb{R},i=1,2\Big\}.$$
Clearly $\Gamma$ is a nonempty, bounded, closed,  convex subset of
the Banach space $B_\sigma(\mathbb{R},\mathbb{R}^2)$.

We are going to show that $F$ maps $\Gamma$ into itself, and is completely comtinuous. Then the Schauder fixed point theorem will yield a fixed point of $F$ in $\Gamma$, and we will then show that
it satisfies  \eqref{function1} and  \eqref{function2a2}.

\begin{lem}\lbl{s3} {\rm (i)} $F(\hat{\varphi})(s)\leq F(\tilde{\varphi})(s)$ for $s\in\mathbb{R}$ if $\hat{\varphi}\leq \tilde{\varphi}$
and $\hat{\varphi}, \tilde{\varphi}\in\Gamma$;

{\rm (ii)} $F_1(\varphi)(s)$ and $F_2(\varphi)(s)$ are nondecreasing in $s\in\mathbb{R}$ for any $\varphi\in\Gamma$.
\end{lem}

\begin{proof}\, We show that $F_1$ satisfies (i) and (ii) stated in the lemma.

Since $F_1(\varphi)(s)=0$ for $s\leq0$, we only need to consider the case of $s>0$. In view of $\lambda_{1}<0<\lambda_{2}$, it is easy to see that
$K_{1}(\xi,s)>0$ for $0<\xi<s$ and $K_{2}(\xi,s)>0$ for $s<\xi$. Thus, by \eqref{l01} and the hypothesis $(\textbf{A}_2)$ we conclude that $F_1(\hat{\varphi})(s)\leq F_1(\tilde{\varphi})(s)$ for $s\in\mathbb{R}$ if $\hat{\varphi}\leq \tilde{\varphi}$
and $\hat{\varphi}, \tilde{\varphi}\in\Gamma$. This proves (i) for $F_1$.

We next consider (ii).
For $\varphi=(\varphi_1,\varphi_2)\in\Gamma$, the hypothesis $(\textbf{A}_2)$ implies that $H_1(\varphi)$ is nondecreasing in $\varphi_i$.
Since $\varphi_1$ and $\varphi_2$ are nondecreasing in $\mathbb{R}$, we have $\varphi_i(s+\theta)\geq\varphi_i(s)$ for
 $\theta>0$ and $i=1, 2$.
This leads to $H_1(\varphi)(s+\theta)-H_1(\varphi)(s)\geq0$. A direct computation gives
$$\begin{aligned}
& F_1(\varphi)(s+\theta)-F_1(\varphi)(s)\\
&=\frac{1}{d_1(\lambda_{2}-\lambda_{1})}\left[\int_0^{s+\theta}K_1(\xi,s+\theta)
H_1(\varphi)(\xi)d\xi+\int_{s+\theta}^{\infty}K_2(\xi,s+\theta)H_1(\varphi)(\xi)d\xi\right]\\
& \hspace{3cm} -\frac{1}{d_1(\lambda_{2}-\lambda_{1})}\left[\int_0^{s}K_1(\xi,s)H_1(\varphi)(\xi)d\xi
+\int_{s}^{\infty}K_2(\xi,s)H_1(\varphi)(\xi)d\xi\right]\\
&=\frac{1}{d_1(\lambda_{2}-\lambda_{1})}\left[\int_0^{s+\theta}e^{\lambda_{1}(s+\theta-\xi)}H_1(\varphi)(\xi)d\xi
+\int_{s+\theta}^{\infty}e^{\lambda_{2}(s+\theta-\xi)})H_1(\varphi)(\xi)d\xi\right.\\
&\hspace{4cm} -\int_0^{s}e^{\lambda_{1}(s-\xi)}H_1(\varphi)(\xi)d\xi
-\left.\int_{s}^{\infty}e^{\lambda_{2}(s-\xi)}H_1(\varphi)(\xi)d\xi\right]\\
&\hspace{0.5cm} +\frac{1}{d_1(\lambda_{2}
-\lambda_{1})}\left[\int_0^{s}e^{\lambda_{1}s}e^{-\lambda_{2}\xi}H_1(\varphi)(\xi)d\xi
+\int_{s}^{\infty}e^{\lambda_{1}s}e^{-\lambda_{2}\xi}H_1(\varphi)(\xi)d\xi\right.\\
&\hspace{4cm} -\left.\int_0^{s+\theta}e^{\lambda_{1}(s+\theta)}e^{-\lambda_{2}\xi}H_1(\varphi)(\xi)d\xi
-\int_{s+\theta}^{\infty}e^{\lambda_{1}(s+\theta)}e^{-\lambda_{2}\xi}H_1(\varphi)(\xi)d\xi
\right]\\
&=\frac{1}{d_1(\lambda_{2}-\lambda_{1})}\left[\int_0^{s+\theta}e^{\lambda_{1}(s+\theta-\xi)}H_1(\varphi)(\xi)d\xi
+\int_{s+\theta}^{\infty}e^{\lambda_{2}(s+\theta-\xi)})H_1(\varphi)(\xi)d\xi\right.\\
&\hspace{5cm} -\int_0^{s}e^{\lambda_{1}(s-\xi)}H_1(\varphi)(\xi)d\xi
-\left. \int_{s}^{\infty}e^{\lambda_{2}(s-\xi)}H_1(\varphi)(\xi)d\xi\right]\\
&\hspace{0.5cm}+\frac{1}{d_1(\lambda_{2}-\lambda_{1})}e^{\lambda_{1}s}\left(1-e^{\lambda_{1}\theta}\right)
\int_0^{\infty}e^{-\lambda_{2}\xi}H_1(\varphi)(\xi)d\xi\\
&\geq\frac{1}{d_1(\lambda_{2}-\lambda_{1})}\left[\int_0^{s+\theta}e^{\lambda_{1}(s+\theta-\xi)}H_1(\varphi)(\xi)d\xi
+\int_{s+\theta}^{\infty}e^{\lambda_{2}(s+\theta-\xi)})H_1(\varphi)(\xi)d\xi\right.\\
&\hspace{5cm} -\int_0^{s}e^{\lambda_{1}(s-\xi)}H_1(\varphi)(\xi)d\xi
      -\left.\int_{s}^{\infty}e^{\lambda_{2}(s-\xi)}H_1(\varphi)(\xi)d\xi\right]\\
&=\frac{1}{d_1(\lambda_{2}-\lambda_{1})}\left\{\int_0^{\theta}e^{\lambda_{1}(s+\theta-\xi)}
H_1(\varphi)(\xi)d\xi+\left[\int_0^{s}e^{\lambda_{1}(s-\xi)}
\Big(H_1(\varphi)(\xi+\theta)-H_1(\varphi)(\xi)\Big)d\xi\right.\right.\\
&\quad\hspace{7.5cm}+\left.\left.\int_{s}^{\infty}e^{\lambda_{2}(s-\xi)}\Big(H_1(\varphi)(\xi+\theta)-H_1(\varphi)(\xi)\Big)d\xi\right]\right\}\\
&\geq 0.
\end{aligned}$$
So $F_2$ satisfies (ii).

Similarly, we can prove $F_2$ satisfies  (i) and (ii).
\end{proof}

\begin{lem}\lbl{s4}
$F(\Gamma)\subset\Gamma$.
\end{lem}

\begin{proof}\,Due to Lemma \ref{s3}, it suffices to show that, for all $s\in\mathbb R$,
$$
\underline{\varphi}(s)\leq F(\underline{\varphi})(s),\; F(\overline{\varphi})(s)\leq\overline{\varphi}(s).
$$
We firstly show
\[
\underline{\varphi}_1(s)\leq F_1(\underline{\varphi})(s),
\;\forall s\in\mathbb R.
\]
 Since $\underline\varphi_1(s)=F_1(\varphi)(s)=0$ for $s\leq0$, we only need to consider the case of $s>0$. Without loss of generality, we denote $\xi_0=0,\xi_{m_1+1}=\infty$ and assume $\xi_i<\xi_{i+1}$ for $i=0,1,2,\cdots,m_1$. Here $\xi_i$, $i\in \{0,..., m_1\}$, are points in $\Omega_1$ so that $\underline \varphi_1$ satisfies the first inequality \eqref{lower} in $\mathbb R^+\setminus\Omega_1$. According to the definition of $F_1(\varphi)$ and Definition \ref{s2}, for any $s\in(\xi_i,\xi_{i+1})$, we have,
$$\begin{aligned}
F_1(\underline{\varphi})(s)&=\frac{1}{d_1(\lambda_{2}-\lambda_{1})}\left[\int_0^{s}K_1(\xi,s)H_1(\underline{\varphi})(\xi)d\xi
+\int_{s}^{\infty}K_2(\xi,s)H_1(\underline{\varphi})(\xi)d\xi\right]\\
&\geq\frac{1}{d_1(\lambda_{2}-\lambda_{1})}\left[\int_0^{s}K_1(\xi,s)\left(\beta\underline{\varphi}_1(\xi)-d_1\underline{\varphi}_1''(\xi)
+c\underline{\varphi}_1(\xi)'\right)d\xi\right.\\
&\quad\hspace{3cm}+\left.\int_{s}^{\infty}K_2(\xi,s)\left(\beta\underline{\varphi}_1(\xi)-d_1\underline{\varphi}_1''(\xi)
+c\underline{\varphi}_1'(\xi)\right)d\xi\right]\\
&=\underline{\varphi}_1(s)+\frac{1}{\lambda_{2}
-\lambda_{1}}\left[\sum_{j=1}^{i}K_1(\xi_j,s)\left(\underline{\varphi}_1'(\xi_j+)
-\underline{\varphi}_1'(\xi_j-)\right)\right.\\
&\quad\hspace{3.5cm}+\left.\sum_{j=i+1}^{m_1}K_2(\xi_j,s)\left(\underline{\varphi}_1'(\xi_j+)
-\underline{\varphi}_1'(\xi_j-)\right)\right]\\
&\geq\underline{\varphi}_1(s).
\end{aligned}$$
The continuity of $\underline{\varphi}(s)$ and $F_1(\underline{\varphi})(s)$ implies that $F_1(\underline{\varphi})(s)\geq\underline{\varphi}_1(s)$ for any $s\in\mathbb{R}^+$.

The proofs of
$$
F_1(\overline{\varphi})(s)
\leq\overline{\varphi}_1(s),
\;
\underline{\varphi}_2(s)\leq F_2(\underline{\varphi})(s),\; F_2(\overline{\varphi})(s)\leq \overline{\varphi}_2(s)
$$
for  $s\in\mathbb{R}$ are similar, and we omit the details.
\end{proof}

\begin{lem}\lbl{s5}
$F: \Gamma\rightarrow \Gamma$ is continuous.
\end{lem}
\begin{proof}\, From the hypothesis $(\textbf{A}_3)$, it is easy to see that, for some $L>0$ and all $\hat\varphi,\;\tilde\varphi\in \Gamma$,
$$|H_1(\hat{\varphi})-H_1(\tilde{\varphi})|_\sigma\leq (L+\beta)|\hat{\varphi}-\tilde{\varphi}|_\sigma.$$
By a direct calculation, we have
$$\begin{aligned}
&|F_1(\hat{\varphi})-F_1(\tilde{\varphi})|_\sigma\\
=& \left|\frac{1}{d_1(\lambda_{2}-\lambda_{1})}\left[\int_0^{s}K_1(\xi,s)(H_1(\hat{\varphi})(\xi)-H_1(\tilde{\varphi})(\xi))d\xi
+\int_{s}^{\infty}K_2(\xi,s)(H_1(\hat{\varphi})(\xi)-H_1(\tilde{\varphi})(\xi))d\xi\right]\right|_\sigma\\
\leq& \frac{1}{d_1(\lambda_{2}-\lambda_{1})}\left|\int_0^{s}K_1(\xi,s)|H_1(\hat{\varphi})(\xi)-H_1(\tilde{\varphi})(\xi)|d\xi
+\int_{s}^{\infty}K_2(\xi,s)|H_1(\hat{\varphi})(\xi)-H_1(\tilde{\varphi})(\xi)|d\xi\right|_\sigma\\
\leq&\frac{L+\beta}{d_1(\lambda_{2}-\lambda_{1})}|\hat{\varphi}-\tilde{\varphi}|_\sigma\sup_{s\in\mathbb{R}^+}\left[\int_0^{s}K_1(\xi,s)
e^{\sigma(\xi-s)}d\xi+\int_{s}^{\infty}K_2(\xi,s)e^{\sigma(\xi-s)}d\xi\right]\\
=&\frac{L+\beta}{d_1(\lambda_{2}-\lambda_{1})}|\hat{\varphi}-\tilde{\varphi}|_\sigma\sup_{s\in\mathbb{R}^+}\left(1
-e^{(\lambda_{1}-\sigma) s}\right)\frac{\lambda_{2}-\lambda_{1}}{(\lambda_{2}-\sigma) (\sigma-\lambda_{1})}\\
\leq&\frac{L+\beta}{d_1(\lambda_{2}-\sigma)(\sigma-\lambda_{1})}|\hat{\varphi}-\tilde{\varphi}|_\sigma,
\end{aligned}$$
which clearly implies $F_1: \Gamma\rightarrow B_\sigma(\mathbb{R},\mathbb{R}^2)$ is continuous.

Similarly we can show  $F_2: \Gamma\rightarrow B_\sigma(\mathbb{R},\mathbb{R}^2)$ is continuous. Hence $F$ is continuous on $\Gamma$.
\end{proof}

\begin{lem}\lbl{s6}
 $F:\Gamma\rightarrow\Gamma$ is compact.
\end{lem}

\begin{proof}\,Since $F$ is continuous on $\Gamma$ by Lemma \ref{s5}, and $\Gamma$ is a bounded set in $B_\sigma(\mathbb R, \mathbb R^2)$,
it suffices to show that $F(\Gamma)$ is a relatively compact set. To this end, let
$$
\rho:=\sup\{|H_i(\varphi)|:\varphi\in\Gamma,i=1,2\}.
$$
In view of $F_1(\varphi)(s)=0$ for $s<0$, we get $(F_1(\varphi))'(s)=0$ when $s<0$. Moreover, for any $\varphi\in\Gamma$ and $s>0$,
\bess(F_1(\varphi))'(s)=\frac{1}{d_1(\lambda_{2}-\lambda_{1})}\left[\int_0^{s} K_{1s}(\xi,s)H_1(\varphi)(\xi)d\xi
+\int_s^{\infty}K_{2s}(\xi,s)H_1(\varphi)(\xi)d\xi\right].\lbl{s001}\eess
Hence,
\bes\begin{aligned}
|(F_1(\varphi))'(s)|&\leq\frac{\rho}{d_1(\lambda_{2}-\lambda_{1})}
\left[\int_{0}^s|K_{1s}(\xi,s)|d\xi
+\int_s^{\infty}|K_{2s}(\xi,s)|d\xi\right]\\
&=\frac{\rho}{d_1(\lambda_{2}-\lambda_{1})}e^{\lambda_{1}s}\left(1-\frac{\lambda_1}{\lambda_2}\right)\leq\frac{\rho}{d_1\lambda_{2}},\;\forall s>0.
\end{aligned}\lbl{s002}\ees
We thus see that $s\to F_1(\varphi)(s)$ is  Lipschitz continuous with Lipschitz constant $L_1:=\frac{\rho}{d_1\lambda_{2}}$ independent
of $\varphi\in\Gamma$.

 Similarly, from \eqref{l02} we have
\bess
(F_2(\varphi))'(s)&=\frac{1}{d_2(\mu_2-\mu_1)}\left[\mu_1\int_{-\infty}^se^{\mu_1(s-\xi)}H_2(\varphi)(\xi)d\xi +\mu_2\int_s^{\infty}e^{\mu_2(s-\xi)}H_2(\varphi)(\xi)d\xi\right],
\eess
and
\bess\begin{aligned}
|(F_2(\varphi))'(s)|&\leq\frac{\rho}{d_2(\mu_2-\mu_1)}\left[|\mu_1|\int_{-\infty}^se^{\mu_1(s-\xi)}d\xi +\mu_2\int_s^{\infty}e^{\mu_2(s-\xi)}d\xi\right]\\
&\leq\frac{2\rho}{d_2(\mu_2-\mu_1)}.
\end{aligned}\lbl{s005}\eess
Thus $\{F(\varphi)(s): \varphi\in\Gamma\}$ is a family of equi-continuous functions of $s\in\mathbb R$.

Let $\Phi_j$ be a sequence of $\Gamma$ and $\upsilon_j=F(\Phi_j)$. Then the sequence $\upsilon_j$ is equi-continuous.
It follows from Lemma \ref{s3}(ii) that $\upsilon_j(s)$ is nondecreasing in $s\in\mathbb{R}$.
Noting that $\Gamma$ is bounded in $L^\infty(\mathbb{R},\mathbb{R}^2)$,
by the Arzela-Ascoli theorem, we conclude that for any $R>0$, there exists a  convergent subsequence of $\upsilon_j|_{[-R,R]}$ in $C([-R,R],\mathbb{R}^2)$. Using a standard diagonal selection scheme, we can extract a subsequence $\upsilon_{j_k}$ that converges in $C([-R,R],\mathbb{R}^2)$ for every $R>0$. Without loss of generality, we assume that the sequence $\upsilon_j$ itself converges in each  $C([-R,R],\mathbb{R}^2)$. From this, it follows easily that
 $\upsilon_j$ is Cauchy in $B_\sigma(\mathbb{R},\mathbb{R}^2)$, and hence it is convergent. This proves the precompactness of $F(\Gamma)$.
\end{proof}

Since $\Gamma$ is a bounded closed convex set of $B_\sigma(\mathbb R, \mathbb R^2)$, by Lemmas \ref{s4}, \ref{s5} and \ref{s6}, we
can apply Schauder's fixed point theorem to conclude that $F$ has a fixed point $\varphi$ in $\Gamma$, which is a non-decreasing
solution of \eqref{function1}. To complete the proof of Proposition \ref{s7}, it remains to prove the following result.

\begin{lem}\lbl{s8} The fixed point
$\varphi$ obtained above satisfies \eqref{function2a2}.
\end{lem}
\begin{proof}
From $\underline \varphi(s)\leq \varphi(s)\leq \overline\varphi(s)$ and $\underline{\varphi}_1(s)=\overline{\varphi}_1(s)=0$ for $s\leq0$,
 $\underline{\varphi}_2(-\infty)=\overline{\varphi}_2(-\infty)=0$, we obtain $\varphi(-\infty)=(0,0)$. Moreover, due to $0\leq \underline \varphi_1(s)\not\equiv 0$ for $s\in\mathbb R$, we have $0\leq  \varphi_1(s)\not\equiv 0$ for $s\in\mathbb R$. It then follows
from the monotonicity of $\varphi_1(s)$ that
$\varphi_1(\infty)\in (0, k_1]$. Using \eqref{function1}, it is well known that (cf. lemma 2.2 in \cite{WZ}) $f_1(\varphi(\infty))=f_2(\varphi(\infty))=0$. Thus we may use
 $(\textbf{A}_1)$ to conclude that $\varphi(\infty)={\bf K}$. Hence  \eqref{function2a2} holds.\end{proof}

\section*{Acknowledgments}
The research in this work was supported by the Natural Science Foundation of China(11671243, 61672021), the Shaanxi New-star Plan of Science and Technology(2015KJXX-21),
the Natural Science Foundation of Shaanxi Province(2014JM1003), the Fundamental Research Funds for the Central Universities(GK201701001, GK201302005), and
 the Australian Research Council.

\end{document}